\documentclass[12pt,twoside,reqno]{amsart}
\allowdisplaybreaks
\usepackage{bm}

\usepackage{amssymb,amsmath,mathrsfs,amstext,amsthm,amsfonts,amscd,xcolor, tikz-cd}

\usepackage{bbm}
\usepackage{dsfont}
\usepackage{tablefootnote}

\usepackage{mathrsfs}

\usepackage{amsthm, enumerate}

\usepackage[ansinew]{inputenc}
\usepackage{graphicx}
\usepackage{hyperref}

\usepackage{footnote} 
\makesavenoteenv{tabular}

\newcommand{\R}{\mathbb{R}}

\newcommand{\N}{\mathbb{N}}
\newcommand{\Z}{\mathbb{Z}}

\newcommand{\SL}{{\rm SL}}

\newcommand{\GL}{\rm GL}

\newcommand{\Mat}{{\rm Mat}}

\newcommand{\Ecal}{\mathcal{E}}

\newcommand{\Leb}{{\rm Leb}}

\newcommand{\sabs}[1]{\left| #1 \right|} 
\newcommand{\abs}[1]{\bigl| #1 \bigr|} 
\newcommand{\norm}[1]{\lVert#1\rVert} 
\newcommand{\normtwo}[1]{
	{\left\vert\kern-0.25ex\left\vert\kern-0.25ex\left\vert #1
		\right\vert\kern-0.25ex\right\vert\kern-0.25ex\right\vert} }

\newcommand{\less}{\lesssim}

\newcommand{\ep}{\epsilon}

\newcommand{\la}{\lambda}

\newcommand{\om}{\omega}
\newcommand{\uom}{\underline{\omega}}
\newcommand{\uomega}{\underline{\omega}}

\newcommand{\cocycles}{\mathcal{M}^\ast}

\newcommand{\mostv}{\mathfrak{v}}

\newcommand{\Pp}{\mathbb{P}}

\newcommand{\Ccal}{\mathcal{C}}

\newcommand{\ind}{\mathds{1}}

\newcommand{\B}{\mathscr{B}}

\newcommand{\Bscr}{\mathscr{B}}
\newcommand{\Nscr}{\mathscr{N}}

\newcommand{\uA}{\underline{A}}
\newcommand{\Ascr}{\mathscr{A}}
\newcommand{\Mscr}{\mathcal{M}}

\newcommand{\Ainv}{\Ascr_{\mathrm{inv}}}
\newcommand{\Asing}{\Ascr_{\mathrm{sing}}}
\newcommand{\Ssing}{\Sigma_{\mathrm{sing}}}
\newcommand{\Sinv}{\Sigma_{\mathrm{inv}}}
\newcommand{\Matdm}{{\rm Mat}^+_2(\R)}

\newcommand\restr[2]{{
		\left.\kern-\nulldelimiterspace 
		#1 
		\vphantom{\big|} 
		\right|_{#2} 
}}

\theoremstyle{plain}
\newtheorem{theorem}{Theorem}[section]
\newtheorem{proposition}{Proposition}[section]
\newtheorem{corollary}{Corollary}[section]
\newtheorem{lemma}{Lemma}[section]

\theoremstyle{definition}
\newtheorem{definition}{Definition}[section]

\theoremstyle{definition}
\newtheorem{remark}{Remark}[section]
\newtheorem{example}[theorem]{Example}
\numberwithin{equation}{section}

\newcommand{\Prob}{\mathrm{Prob}}

\newcommand{\rank}{\mathrm{rank}}

\newcommand{\Range}{\mathrm{Range}}

\newcommand{\Qop}{\mathcal{Q}}

\newcommand{\Qinv}{\Qop_{\rm inv}}
\newcommand{\Qsing}{\Qop_{\rm sing}}

\newcommand{\Kinv}{K_{{\rm inv}}}

\title[Random two dimensional cocycles]{Random 2D linear  cocycles II:\\statistical properties}

\date{}

\begin{document}

\author[P. Duarte]{Pedro Duarte}
\address{Departamento de Matem\'atica and CEMS.UL\\
Faculdade de Ci\^encias\\
Universidade de Lisboa\\
Portugal
}
\email{pmduarte@fc.ul.pt}

\author[M. Dur\~aes]{Marcelo Dur\~aes}
\address{Departamento de Matem\'atica, Pontif\'icia Universidade Cat\'olica do Rio de Janeiro (PUC-Rio), Brazil}
\email{accp95@gmail.com}

\author[T. Graxinha]{Tom\'e  Graxinha}
\address{Departamento de Matem\'atica and CEMS.UL\\
	Faculdade de Ci\^encias\\
	Universidade de Lisboa\\
	Portugal}
\email{graxinhatome@hotmail.com}

\author[S. Klein]{Silvius Klein}
\address{Departamento de Matem\'atica, Pontif\'icia Universidade Cat\'olica do Rio de Janeiro (PUC-Rio), Brazil}
\email{silviusk@puc-rio.br}

\begin{abstract}
Consider the space of two dimensional random linear cocycles over a shift in finitely many symbols, with at least one singular and one invertible matrix.
We provide an explicit formula for the unique stationary measure associated to such cocycles and establish a Furstenberg-type formula characterizing the Lyapunov exponent. Using the spectral properties of the corres\-ponding Markov operator and a parameter elimination argument, we prove that Lebesgue almost every cocycle in this space satisfies large deviations estimates and a central limit theorem.
\end{abstract}

\maketitle


\section{Introduction and statements}\label{intro}
Let $\nu$ be a compactly supported measure on  the space of matrices $\Mat_m (\R)$, $m\ge2$.
Consider an i.i.d. sequence of matrices $\{A_n\}_{n\in\Z}$ in  $\Mat_m (\R)$, with common distribution $\nu$, and the corresponding multiplicative process $\Pi_n := A_n \cdots A_2 \, A_1$. Under a general integrability assumptions, namely $\int \log^+ \norm{A} d \nu (A) < \infty$, by Furstenberg-Kesten's theorem we have that
\begin{equation}\label{def L1}
\lim_{n\to\infty} \frac 1 n \log \norm{\Pi_n} =: L_1 (\nu)
\end{equation}
exists almost surely and it is a constant. We call this constant $L_1 (\nu)$ the first Lyapunov exponent of the multiplicative process determined by the measure $\nu$. 

Some of the central problems regarding Lyapunov exponents of such processes concern the regularity of the map $\nu \mapsto L_1 (\nu)$, as well as the availability of other probability laws. More specifically, the almost sure convergence (which implies convergence in probability)  in~\eqref{def L1} is the multiplicative analogue of the law of large numbers. It is then natural to ask if there is an explicit rate of convergence in probability, that is, if a large deviations type (LDT) estimate holds. Moreover, it is also natural to ask if a central limit theorem (CLT) holds as well.

Multiplicative processes as described above appear naturally in the study of the discrete Schr\"odinger operator with random potential on the integer lattice (or, more generally, on a band-integer lattice). 

More precisely, let $\{v_n\}_{n\in\Z}$ be an i.i.d. sequence of real-valued random variables and consider the operator $H$ on $l^2 (\Z, \R)$ given by
\begin{equation}\label{S op}
(H u)_n := - (u_{n+1} + u_{n-1}) + v_n \, u_n
\end{equation}
for $u = \{u_n\}_{n\in\Z} \in l^2 (\Z, \R)$.

The Schr\"odinger (or eigenvalue) equation $H u = t \, u$ becomes a first-order matrix recurrence. Its (formal) solution is given by
$$\begin{pmatrix}
u_{n+1}\\
u_{n}
\end{pmatrix} 
=
\begin{pmatrix}
v_n-t & -1\\
1 & 0
\end{pmatrix} 
\cdots
\begin{pmatrix}
v_1-t & -1\\
1 & 0
\end{pmatrix} 
\begin{pmatrix}
u_{1}\\
u_{0}
\end{pmatrix} 
$$
thus leading to a multiplicative process in $\SL_2 (\R)$ generated by the i.i.d. sequence of random matrices $\{A_n\}_{n\in\Z}$ where 
$A_n = \begin{pmatrix}
v_n-t & -1\\
1 & 0
\end{pmatrix}$.

More general operators than~\eqref{S op} (e.g. Jacobi operators) or band lattice versions thereof lead to multiplicative processes in other, more general semigroups of matrices. Moreover, the properties of their Lyapunov exponents as discussed above (especially their H\"older continuity and the availability of large deviations) are directly relevant in the study of the spectral properties of the corresponding operator. See the survey~\cite{David-survey} for a nice review of these connections.  

\medskip

Multiplicative processes of this or a more general type (i.e. non independent) can be studied in the abstract framework of linear cocycles in ergodic theory. 
A base measure preserving dynamical system $(X,\mu,f)$, where $X$ is a compact metric space, $f \colon X \to X$ is continuous transformation and $\mu$ is an $f$-invariant, ergodic Borel probability measure on $X$, together with a bounded, measurable fiber map $A \colon X\to \Mat_m (\R)$ determine the skew-product transformation $F \colon X\times\R^m \to X\times\R^m$,
$$F(x, v)=(f (x), \, A(f x) v) \, .$$
The new dynamical system $F$ is called a linear cocycle. Its iterates are
$F^n(x, v)=(f^n (x), \, A^n (x) v)$, where for all $n\in\N$, 
$$A^n (x):= A(f^{n} x)  \cdots A(f^2 x) \,  A(f x) \, .$$

The first Lyapunov exponent of a linear cocycle $F$ measures the asymptotic exponential growth of its fiber iterates. It is defined via the Furstenberg-Kesten theorem as the $\mu$-a.e. limit
$$ L_1 (F) = L_1 (A) := \lim_{n\to\infty} \, \frac{1}{n} \, \log \norm{ A^n (x) } \, $$
provided  $A$ satisfies the integrability condition $\int_X \log^+ \norm{A}  d \mu < \infty$.

Moreover, if the norm (that is, the first singular value) of the fiber iterates is replaced by the second singular value, the corresponding $\mu$-a.e. limit above still exists, it is called the second Lyapunov exponent and it is denoted by $L_2 (F) = L_2 (A)$.

\smallskip

An i.i.d. multiplicative process in $\Mat_m (\R)$  can be identified with a locally constant linear cocycle over a Bernoulli shift. Indeed, let $\Ascr$, the space of symbols (or alphabet),  be a compact metric space  and let $p$ be a probability measure on $\Ascr$. Denote by $X:=\Ascr^\Z$ the space of bi-infinite sequences $\om = \{\om_n\}_{n\in\Z}$ in this alphabet, which we endow with  the product measure $\mu = p^\Z$. Let $\sigma \colon X \to X$ be the corresponding forward shift $\sigma \om = \{\om_{n+1}\}_{n\in\Z}$. Then $(X, \mu, \sigma)$ is a measure preserving, ergodic dynamical system called a Bernoulli shift.  A continuous function $\Ascr \ni \om_0 \mapsto A_{\om_0} \in  \Mat_m (\R)$ determines the locally constant fiber map $A \colon X \to \Mat_m (\R)$, $A (\om) = A_{\om_0}$, which in turn determines a linear cocycle over the Bernoulli shift, referred to as a (random) Bernoulli cocycle. 

In this paper we restrict to the case of a {\em finite} alphabet and two-dimensional matrices. That is, $\Ascr =\{1,\ldots, k\}$ for some $k\ge2$, $p = (p_1, \ldots, p_k)$ is a probability vector with $p_i > 0$ for all $i$, while the fiber map $A$ is determined by a $k$-tuple $\underline{A} = (A_1, \ldots, A_k) \in \Mat_2 (\R)^k$. We identify the corresponding Bernoulli cocycle with the tuple $\uA$ and
denote by $L_i (\underline{A}) = L_i (\underline{A}, p)$, $i \in \{ 1,2 \}$, its Lyapunov exponents.

\smallskip

More generally, we also consider linear cocycles over a Markov shift. That is, let $P$ be a  (left) stochastic matrix, i.e. $P=(p_{i j})_{1\leq i, j\leq k}$ with $p_{ij}\geq 0$
and $\sum_{i=1}^k p_{i j} = 1$ for all $1\le j \le k$.
Given a $P$-stationary probability vector $q$, i.e.,
$q=P\, q$, the pair $(P, q)$ determines a probability measure $\mu$ on $X$  for which the process
$\xi_n \colon X\to \Ascr$, $\xi_n(\omega):=\omega_n$ is
a stationary Markov chain in $\Ascr$ with probability transition 
matrix $P$ and initial distribution law $q$. 
Then 
$(X, \sigma, \mu)$ is a measure preserving dynamical system called a Markov shift.

Given a finite word $(i_0, i_1,\ldots, i_n)\in \Ascr^{n+1}$ and $k\in\Z$, the set 
$$[k; i_0, i_1, \ldots, i_n] := \left\{ \om \in X  \colon \om_{k+l} = i_l \quad  \text{for all } 0 \le l \le n \right\} $$
is called a cylinder of length $n+1$ in $X$.
Its (Markov) measure is then
$$\mu \left( [k; i_0, i_1, \ldots, i_n] \right) = q_{i_0} \, p_{i_i, i_0} \, \cdots \, p_{i_n, i_{n-1}} \, .$$


We will assume that the matrix $P$ is primitive, i.e. there exists $n\ge1$ such that  $p^n_{i j}>0$ for all entries of the power matrix $P^n$. Then
$\lim_{n\to \infty} p^n_{i j} = q_i > 0$
for all $1\le i, j \le k$ and the corresponding Markov shift $(X, \mu, \sigma)$ is ergodic and mixing, see~\cite[Theorems 7.2.8, 7.2.11]{FET}.

As before, a $k$-tuple $\underline{A} = (A_1, \ldots, A_k) \in \Mat_2 (\R)^k$ determines a locally constant linear cocycle over this base dynamics, which we refer to as a (random) Markov cocycle. Its Lyapunov exponents are denoted by $L_i (\underline{A}) = L_i (\underline{A}, P, q)$.

\smallskip

When we restrict to {\em invertible} random linear cocycles, that is, when $\underline{A} = (A_1, \ldots, A_k) \in \GL_2 (\R)^k$, the problems on Lyapunov exponents formulated above are well understood, having been the subject of many works throughout the years. See Table~\ref{table} below for a  more detailed picture of available results. The case when all matrices $A_1, \ldots, A_k$ defining the cocycle $\uA$ are singular (noninvertible) is  treated  easily.

\smallskip

We thus consider the case of random cocycles with {\em both singular and invertible} components, and assume, moreover, that the invertible components have positive determinant. In a recent work, see~\cite{DDGK-paper1}, we show that the first  Lyapunov exponent of such cocycles behaves rather differently from the invertible case:  when it comes to its regularity, it exhibits a dichotomy in the spirit of Bochi-Ma\~n\'e's, from being analytic to being discontinuous, see~\cite[Theorem 1.1 and Corollary 1.1]{DDGK-paper1}. 

\smallskip

As explained in~\cite{DDGK-paper1}, random cocycles with both singular and inverti\-ble components can be used to provide a formal model  of a one-parameter family of Bernoulli cocycles $\uA_t \in  \Mat_2 (\R)^2$ associated to the random Schr\"odinger operator $H$ whose potential takes two values: $\infty$ and a finite value $a$, as in Craig and Simon~\cite[Example 3]{CraigSimon}.\footnote{Other versions of such an operator (with more values) can similarly be considered.} 

Indeed, given a probability vector $(p, 1-p)$ with $p\in (0, 1)$ and $a\in\R$, define the random cocycle $\uA_t =(A_t(1), A_t(2)) \in \Mat_2(\R)^2$ by $A_t(1) \equiv \begin{bmatrix}
	1 & 0 \\ 0 & 0 
\end{bmatrix}$ and
$A_t(2)=\begin{bmatrix}
	a-t & -1 \\ 1 & 0 
\end{bmatrix}$. 

Moreover, let $H_\la$ be the Schr\"odinger operator with random potential $\{v_j\}_{j\in\Z}$ taking the value $\la$ with probability $p$ and the value $a$ with probability $1-p$, and denote by $S_{t, \la}(j) = \begin{bmatrix}
	v_j-t &-1\\  1 & 0 	
\end{bmatrix} $ its corres\-ponding random Schr\"odinger cocycle. 

Then $\uA_t$ is the 
asymptotic limit as $\la\to\infty$ of the {\em rescaled} random Schr\"odinger cocycle 
$\uA_{t,\la}=(\uA_{t,\la}(1), \uA_{t,\la}(2))$ where $\uA_{t,\la}(j)=\la_j^{-1}\, S_{t, \la} (j)$ and $\la_j =\la$ or $1$ with probability $p, 1-p$. 
Note the following relation between the Lyapunov exponents:  $L_1 (S_{t, \la})= p \log \la + L_1 (\uA_{t,\la})$.

For each of the limiting cocycles $\uA_t$, one component is invertible and the other one is singular. The map $t\mapsto L_1 (\uA_t)$ is a.e. discontinuous on $[a-2, a+2]$, the limit (as $\la\to\infty$) of the spectrum of the operator $H_\la$.  This  correlates well with the fact (proven by Craig and Simon) that the integrated density of states of the operator $H$  is discontinuous. See ~\cite[Example 5.4]{DDGK-paper1} for more details.

%

\medskip

The main purpose of this paper is to establish statistical properties (large deviations type estimates and a central limit theorem) for random (Bernoulli and Markov) cocycles with both singular and invertible components. 
Let $\Ascr = \Asing \sqcup \Ainv$ be a partition into two nonempty sets of the alphabet and let 
\begin{align*}
\Mscr :=  \big\{  \uA = (A_1, \ldots, A_k) \in \Mat_2(\R)^\Ascr  \colon  & \rank A_i = 1 \quad \forall i\in\Asing ,  \\
   & \det(A_j) >  0 \quad  \forall j\in\Ainv  \big\} .
\end{align*}
The set $\Mscr$ is an analytic manifold (see the end of Section 4 in~\cite{DDGK-paper1}), thus it is equipped with a Lebesgue measure.

\begin{theorem}
\label{thm: ldt intro}
For Lebesgue almost every random cocycle $\uA\in \Mscr$, 
$L_1(\uA)> - \infty$ and  for every  $\varepsilon>0$ we have
$$ \mu \left\{ \omega\in X \colon \abs{ \frac{1}{n}\, \log \norm{A^n(\omega)}-L_1(\uA) }>\varepsilon   \right\}\le C \,  e^{-c_0 (\varepsilon)\, n^{1/3}} $$
where $C< \infty$, $c_0 (\varepsilon)>0$ is an explicit function of $\varepsilon$ and $\mu$ is the Bernoulli (or Markov) measure on $X$ defined above.
\end{theorem}

For the next result, in the Markov case only, we assume that the subshift of finite type $\uA_{\mathrm{inv}}$, defined by the same underlying graph but retaining only edges labeled by invertible matrices in $\uA$,  is topologically mixing. 
This technical condition allows the use of Avila--Bochi--Yoccoz theory via the work in~\cite{DDGK-paper1}. 
As in that work, the assumption is purely technical: it can be removed here, in~\cite{DDGK-paper1}, as well as in Avila--Bochi--Yoccoz, where uniform hyperbolicity is characterized by invariant multi-cones. 
For simplicity, we impose it and leave the general Markov case for future work. 
In the Bernoulli case, this assumption is unnecessary.

\begin{theorem}
\label{thm: clt intro}
For Lebesgue almost every random cocycle $\uA\in \Mscr$, 
$L_1(\uA)> - \infty$ and there exists $\sigma>0$ such that  the following convergence in distribution to the normalized Gaussian holds:
$$ \frac{ \log \norm{A^n}-n\,L_1(\uA) }{\sigma \, \sqrt{n}} \stackrel{d}{\to}  \Nscr(0,1)  \, .$$
\end{theorem}

We will actually prove something stronger. Given any random co\-cycle $\uA\in \Mscr$ with no null words, if $t \mapsto \uA_t$ is a one-parameter family of such cocycles passing through $\uA$ and satisfying a certain monoto\-ni\-ci\-ty property (positive winding), then for almost every parameter $t$, the cocycle $\uA_t$ satisfies LDT estimates and a CLT.

The main tools in deriving such limit laws for random linear cocycles and other types of dynamical systems are the spectral properties of either the Markov  transition operator  or the Ruelle transfer operator associated to the system. This spectral method goes back to the work of Nagaev~\cite{Nagaev} on CLTs for stationary Markov chains and it is described in an abstract setting in the book~\cite{Hennion-Herve} of Hennion and Herv\'e. The main ingredient required by the method is the quasi-compactness (which is equivalent to the spectral gap property) of the transition/transfer operator on an appropriate space of observables. 

Recently, Cai, Duarte, Klein~\cite{CDK-paper3} obtained an {\em effective} abstract LDT estimate for Markov chains under a much weaker hypothesis, a certain mixing property. Theorem~\ref{thm: ldt intro} will be established using this result, its effectiveness being  crucial in our argument. To establish Theorem~\ref{thm: clt intro} we will use an abstract CLT due to Gordin-Liv\v sic~\cite{Go78}. 

It turns out that these abstract results are not directly applicable to cocycles $\uA\in \Mscr$, precisely because of their singularities. We will employ some rather involved parameter elimination arguments that eventually lead to limit laws for almost every parameter. 

Moreover, we will also derive an explicit formula for the (unique) stationary measure of the transition operator (see Proposition~\ref{prop stationary measure}) and a Furstenberg-type formula for the first Lyapunov exponent (see Theorem~\ref{Furstenberg's Formula}). These results are crucial in the arguments used for establi\-shing the limit laws and also interesting in themselves. We note, moreover, that standard results and methods of Furstenberg's theory (such as the formula describing the first Lyapunov exponent via stationary measures) are {\em not} applicable in our setting of singular cocycles, and will be studied independently in this work. 


Together with the results in our first paper~\cite{DDGK-paper1} on this topic, as well as the previously available works on invertible matrix components, we are now able to present an almost complete picture regarding the crucial problems on Lyapunov exponents described above, in the setting of random cocycles driven by a finitely supported measure on $\Mat_2 (\R)$.

\smallskip

The table below thus summarizes what it is known regarding the minimal regularity of the Lyapunov exponent (R-LE), namely its modulus of continuity\footnote{Given a metric space $(M, d)$, a function $\phi\colon M \to \R$ is said to be weak-H\"older continuous if $\abs{\phi (x) - \phi (y)} \le C  \exp \left(- \alpha \,  \log^b \frac{1}{d (x, y)} \right)$ for some $C<\infty, \alpha, b \in (0, 1]$ and all $x, y \in M$. When $b=1$, this corresponds to $\alpha$-H\"older continuity. \\ Moreover, if $\abs{\phi (a) - \phi (x)} \le C \, d (a, x)^\alpha$ holds for a given point $a$ and all $x$, we call $\phi$ pointwise H\"older at $a$.} or whether it is discontinuous (Disc.) as well as the availability of large deviations type (LDT) estimates and of a central limit theorem (CLT) for $\Mat_2 (\R)$-valued  Bernoulli cocycles $\uA \in \Mat_2 (\R)^k$, $k\ge 2$. For the purpose of this table, we assume that $L_1 >  L_2 \geq -\infty$.  There are three possibilities for such a  given cocycle $\uA$: $\rank = 2$, meaning its components are all invertible; $\rank=1$, meaning its components are all singular; $\rank=1\&2$, the case treated in this paper, where some components are invertible and some singular.

\begin{table}[h]\label{table}
	\begin{tabular}{|c||c|c|c|}
		\hline
		& R-LE  & LDT  & CLT \\
		\hline\hline
		$\rank=2$ & (Weak) H\"older\tablefootnote{Locally H\"older for quasi-irreducible cocycles~\cite{LP89},~\cite{DK-book}; locally weak-H\"older in the remaining case~\cite{DK-Holder}; pointwise H\"older always~\cite{Tall-Viana}.} 
		& \quad Yes\tablefootnote{Locally uniform LDT of exponential type in the quasi-irreducible case~\cite{DK-book}; locally uniform LDT of sub-exponential type in the remaining case~\cite{DK-Holder}; non-uniform LDT of exponential type holds always~\cite{DK-Holder}.} \quad
		& \quad Yes\tablefootnote{See~\cite{LP82},~\cite{BQ-CLT} and~\cite{Pablo-Malicet}.}\quad
		\\
		\hline
		$\rank=1$ 
		& $C^\omega$  (\cite[Cor. 4.4]{DDGK-paper1})
		& Yes (Rmk.~\ref{ldt-clt rank=1})
		& Yes (Rmk.~\ref{ldt-clt rank=1})
		\\
		\hline
		$\rank=1\&2$
		& Disc. (\cite[Cor. 1.1]{DDGK-paper1})
		& Yes (Thm.~\ref{thm: ldt intro}) 
		& Yes (Thm.~\ref{thm: clt intro})
		\\
		\hline
	\end{tabular}
	\medskip
	\caption{Random (Bernoulli) two dimensional cocycles.}
	\vspace{-10pt}
\end{table}

The concept of projectively uniformly hyperbolic cocycle (or dominated splitting property) is defined and characterized for this setting in~\cite[Section 2]{DDGK-paper1}, extending results of Avila, Bochi, Yoccoz~\cite{ABY10} for $\SL_2 (\R)$-valued cocycles. Recall that in this  case  the Lyapunov exponent is analytic and the statistical properties hold trivially (they are easily reduced to their classical analogues for additive i.i.d. processes).

Moreover, in this paper we show that the statistical properties on the second and third lines of Table~\ref{table} also hold for mixing {\em Markov} cocycles. Furthermore, in this Markov setting, the results on the first line are available only in the generic (irreducible) case (see~\cite[Chapter 5]{DK-book}), but we expect them to still hold without the generic assumption, as is the case in the Bernoulli setting, see~\cite{DK-Holder}.

\smallskip

The rest of the paper is organized as follows. In Section~\ref{stationary} we show that given a random cocycle with both invertible and singular components, the induced projective Markov chain is a renewal process with a unique and explicitly defined stationary measure supported on a discrete set. Moreover, the transition stochastic kernel is uniformly ergodic in the sense that its iterated kernel converges to the stationary measure in the total variation norm. Furthermore, we establish a Furstenberg-type formula for the Lyapunov exponent. Using parameter elimination arguments, in Section~\ref{StatProp} we establish statistical properties (large deviations and a central limit theorem) for one-parameter families of such cocycles satisfying a positive winding condition. This in particular implies the availability of these limit laws for Lebesgue almost every random cocycle in finite symbols, i.e. Theorems~\ref{thm: ldt intro} and~\ref{thm: clt intro}.

\subsection*{Dedication.} The last named author would like to dedicate this paper  to the memory of Michael Goldstein, whose work has deeply influenced his own. My interest in  problems concerning the regularity of Lyapunov exponents, large deviations estimates, spectral theory of Schr\"odinger operators and the relations between these topics, as well as the subtleties involved in performing  delicate parameter elimination arguments, all stem from my exposure to Michael's profound work (e.g.~\cite{B-G-first, GS01, GS-fine}), which has permeated this and indeed most of my research alone or in collaboration with others, throughout the years. 

\medskip  

\section{Stationary measures}\label{stationary}

\smallskip

In this section we establish formulas for the stationary measure and the Lyapunov exponent of a random linear cocycle $F \colon X\times \R^2\to X\times \R^2$ with both singular and invertible components. 

\smallskip

Given a finite alphabet  $\Ascr = \{1, \ldots, k\}$ and a partition into two nonempty sets $\Ascr = \Asing \sqcup \Ainv$, let $\uA :=(A_i)_{i\in\mathscr{A}} \in  \Mat_2(\R)^k$ be a $k$-tuple such that $\rank A_i = 1$ if $i\in \Asing$ and $\rank A_i = 2$ if $i\in \Ainv$. Note that in this section we do not make any other assumptions on $\uA$. 

Let $P=\{p_{i j}\}_{1\le i, j \le k}$ be a primitive left stochastic matrix (that is, $p_{i j}$ encodes the transition probability from state $j$ to state $i$) and let $q=\{q_i\}_{1\le i\le k}$ be its unique stationary probability vector. 

Then the triplet $(\uA, P, q)$ determines the (Markov) random linear cocycle $F \colon X \times \R^2 \to X \times \R^2$, where $X=\Ascr^\Z$ is endowed with the Markov measure $\mu$ with initial distribution $q$ and transition probability given by $P$ and  $F (\om, v) = (\sigma \om, A(\sigma \om) v)=(\sigma \om, A_{\om_1} v)$. 

\smallskip

Let $\Pp^1:=\Pp(\R^2)$ denote the projective line. An element of $\Pp^1$ will be denoted by $\hat v$, where $v$ is a nonzero vector (or a one dimensional subspace) in $\R^2$.  Given an invertible matrix $A \in \Mat_2 (\R)$, its induced projection action $\hat A \colon \Pp^1 \to \Pp^1$ is given by $\hat A \hat v := \widehat{A v}$.  If $A$ has rank $1$ (that is, if it is nonzero and noninvertible), we define its projective action as the constant map $\hat A \hat v := \hat r$, where $r = \Range (A)$.   Note that we are formally removing the discontinuity of the projective action of a singular matrix at the kernel. 

For a random cocycle  $\uA = (A_i)_{i\in\Ascr}\in\Mat_2(\R)^\Ascr$ we will use the notations $r_i$ and $k_i$ to represent, respectively, the range and the kernel of $A_i$, as well as, when convenient, a unit vector belonging to these one-dimensional subspaces.

\begin{remark}
In the Bernoulli case, if we consider an i.i.d. sequence of $\Ascr$-valued random variables $\{\xi_n\}_n$,  then the sequence of return times to $\Asing$ forms a discrete time renewal process, see~\cite[Definition 2.1]{Barbu2008} for the meaning of this concept. Note that at these return times, when $\xi_n=i\in \Asing$,  the projective action induced by $\uA$ is reset 
to the range of $A_i$. This simple observation lies at the heart of all the explicit formulas presented below. Since in the Markov case the reduction to the setting of~\cite{Barbu2008} is less evident, the results of this section will be  presented for this case. The Bernoulli case can be treated similarly, but in a much more straightforward way; moreover, all relevant quantities have much simpler expressions.
\end{remark}

Let $L^\infty(\Ascr\times\Pp^1)$ be the Banach space of bounded and measurable functions $\varphi \colon \Ascr\times\Pp^1\to\R$ endowed with the sup norm denoted by $\norm{\cdot}_\infty$.

The cocycle $(\uA, P, q)$ determines the operator $\Qop \colon L^\infty(\Ascr\times\Pp^1)\to L^\infty(\Ascr\times\Pp^1)$
defined by
\begin{align*}
 (\Qop \varphi)(j,\hat v) &:= \sum_{ i\in\Ascr}  \varphi(i,\hat A_i\, \hat v)\, p_{ij} \\
& = \sum_{ i\in\Ainv}   \varphi(i, \hat A_i\, \hat v)\,p_{ij}  +  \sum_{ i\in\Asing} \varphi(i, \hat r_i)\,p_{ij} .
\end{align*}
Moreover, we write $\Qop = \Qinv + \Qsing$, where the operators $\Qinv$ and $\Qsing$ are given respectively by the two terms above. 

The operator $\Qop$ is clearly linear, positive and it takes the constant function $\ind$ to itself; in other words, it is a {\em Markov operator}. Let us recall the concept of stationary measure for such an operator.

\begin{definition}\label{def stat measure}
A measure $\eta \in \Prob (\Ascr \times \Pp^1)$ is $\Qop$-stationary if for all observables $\varphi \in  L^\infty(\Ascr\times\Pp^1)$,
$$\int \Qop \varphi \, d \eta = \int \varphi \, d \eta .$$
In this case we also call $\eta$ stationary relative to the cocycle $(\uA, P, q)$.
\end{definition}

By general principles, a stationary measure always exists. In our setting it will be shown to be unique. 

Let $\pi \colon \Ascr\times \Pp^1\to  \Ascr$ denote the canonical projection in the first coordinate. If $\eta \in \Prob (\Ascr \times \Pp^1)$ is $\Qop$-stationary then its push-forward measure via $\pi$ is the $P$-stationary measure $q$ on $\Ascr$, that is, $\pi_\ast\eta=q$.

Indeed, the left stochastic matrix $P$ induces the Markov operator $u \mapsto u \, P$ on $L^\infty (\Ascr) \simeq \R^k$. For a (column) vector $q'$, $P q' = q'$ if and only if $u P q' = u q'$ for every (row) vector $u \in \R^k$, so $q'$ is $P$-stationary if and only if it is stationary relative to its corresponding Markov operator. Moreover, given any $u \in L^\infty (\Ascr)$ and applying the fact that $\eta$ is $\Qop$-stationary to the observable $\varphi := u \circ \pi$ it follows immediately that 
$\displaystyle \int u P \, d (\pi_\ast\eta) = \int u \, d (\pi_\ast\eta) ,$
showing that $\pi_\ast\eta$ is $P$-stationary. But since $P$ is primitive, it has a unique stationary probability vector, so $\pi_\ast\eta = q$.

\smallskip  

Since $\eta$ projects down via $\pi$ to $q$, we can consider 
its disintegration $\{\eta_i\}_{i\in\Ascr} \subset \Prob (\Pp^1)$, which is characterized by 
$$\int_{\Ascr\times\Pp^1}  \varphi (i, \hat v) \, d \eta (i, \hat v) = \sum_{i\in\Ascr} \, q_i \, \int_{\Pp^1} \varphi (i, \hat v) \, d \eta_i (\hat v) \quad \forall \varphi \in L^\infty(\Ascr\times\Pp^1) \, .$$

Then
\begin{equation*}
\eta=\sum_{i\in\Ascr} q_i\, \delta_i\times\eta_i \quad \text{and} \quad \eta_i (E) = \frac{1}{q_i} \, \eta (\{i\} \times E) \,\, \forall i\in \Ascr, E\subset \Pp^1 \text{ Borel} .
\end{equation*}

Given $n\ge 1$, $s \in \Asing$ and $l \in\Ascr$, let $\Bscr_{n}(s,l)$ denote the set of words $\uom=(\omega_0, \ldots, \omega_{n})$
of length $n+1$ such that $\omega_0=s$, $\omega_{n}=l$  and 
$\omega_i\in \Ainv$ for all $i=1, \ldots, n-1$.
For such a word we write
$A^{n}(\uom):=A_{\omega_{n}}\, \ldots\, A_{\omega_2}\, A_{\omega_1}$  and also put
$p (\uom):=p_{\omega_n\,\omega_{n-1}} \, \cdots \, p_{\omega_1\,\omega_{0}}$. 
  
 \smallskip
  
With these notations, we have the following explicit formula for a (or, a-posteriori, {\em the}) $\Qop$-stationary measure. 

\begin{proposition}
\label{prop stationary measure}
Let $\eta=\displaystyle\sum_{j \in \Ascr}q_j \, \delta_j \times \eta_j$
where for all $j\in\Ascr$,

$$
\eta_j := \frac{1}{q_j} \sum_{s \in \Asing}q_s \sum_{n=1}^\infty
\sum_{\uomega \in \Bscr_{n}(s, j)}  p(\uomega)\delta_{\hat A^n(\uomega) \hat r_s}.
$$
Then $\eta$ is an $(\uA, P, q)$-stationary probability measure on $\Ascr\times\Pp^1$.
\end{proposition}

\begin{proof} We first show that for all $j\in\Ascr$,
\begin{equation}
\label{q formula}
	\sum_{s \in \Asing}q_s \sum_{n=1}^\infty
	\sum_{\uomega \in \Bscr_{n}(s,j)}  p(\uomega)=q_j .
\end{equation}
This will then imply that $\eta_j$ is a probability measure on $\Pp^1$ for all $j\in\Ascr$, hence $\eta$ is a probability measure on $\Ascr\times\Pp^1$.

Indeed, given $j\in\Ascr$, consider the cylinder $[0; j]$ determined by the word $(j)$, so $\mu ([0; j]) = q_j$. This cylinder is partitioned, modulo a zero  measure set (mod $0$), into cylinders $[-n; s \uom' j]$, determined by words $(s, \uom', j)$ where $s\in\Asing$ and $\uom' = (\om_1, \cdots, \om_{n-1})$ with $\om_i \in \Ainv$ for all $1\le i \le n-1$, $n\ge 0$ (when $n=1$,   the word $(s, \uom', j)$ is simply $(s, j)$).  That is,
$$[0; j] = \bigsqcup_{s\in\Asing} \, \bigsqcup_{n=1}^\infty \, \bigsqcup_{\om=(s, \uom', j) \in \Bscr_n (s, j)} \, [-n; s \uom' j] \quad (\text{mod } 0) . $$

The cylinders $[-n; s \uom' j]$ are mutually disjoint because $s\in\Asing$ and $\om_i\in\Ainv$ for all $i$. Modulo a zero measure set, their union is $[0; j]$ since, by the mixing of the Markov shift determined by $(P, q)$, almost every infinite word starting with the symbol $j$ contains in its past a symbol $s\in\Asing$. Moreover, 
$\mu ( [-n; s \uom' j] ) = q_s \, p (s, \om', j)$ which establishes the identity~\eqref{q formula}.

It remains to verify that $\eta=\sum_{j \in \Ascr}q_j\delta_j \times \eta_j$ is $\Qop$-stationary, that is, that 
$$\int \Qop \varphi \, d \eta = \int \varphi \, d \eta$$ for all $\varphi \in L^\infty (\Ascr \times \Pp^1)$.

\smallskip

Firstly note that if $j\in\Asing$ then $\eta_j = \delta_{\hat r_j}$. 
Indeed, for all $n\ge 1$, $s\in\Asing$, $\uom \in \Bscr_n (s, j)$ we have
$\hat A^n (\uom) \hat r_s = \hat A_j \cdots \hat A_{\om_1} \hat r_s = \hat r_j$, 
so using the identity~\eqref{q formula},
\begin{align*}
\eta_j &= \frac{1}{q_j} \sum_{s \in \Asing}q_s \sum_{n=1}^\infty
\sum_{\uomega \in \Bscr_{n}(s, j)}  p(\uomega)\delta_{\hat A^n(\uomega) \hat r_s} \\
& =
\frac{1}{q_j} \Big(  \sum_{s \in \Asing}q _s \sum_{n=1}^\infty
\sum_{\uomega \in \Bscr_{n}(s, j)}  p(\uomega) \Big) \, \delta_{\hat r_j} = \delta_{\hat r_j} .
\end{align*}

Secondly, for any observable $\varphi \in L^\infty (\Ascr \times \Pp^1)$, we write
\begin{align*}
\int \varphi \, d \eta = \sum_{j\in\Ascr} q_j \, \int \varphi (j, \hat v) \, d \eta_j (\hat v) = \Sinv (\varphi) + \Ssing (\varphi)
\end{align*}
where
\begin{align*}
\Sinv (\varphi) &:= \sum_{j\in\Ainv} q_j \, \int \varphi (j, \hat v) \, d \eta_j (\hat v) \\
&= \sum_{j\in\Ainv} \, \sum_{s\in\Asing} q_s \, \sum_{n=1}^\infty \sum_{\uom\in\Bscr_n (s, j)} p (\uom) \, \varphi (j, \hat A^n(\uomega) \hat r_s)
\end{align*}
and
\begin{align*}
\Ssing (\varphi) &:= \sum_{j\in\Asing} q_j \, \int \varphi (j, \hat v) \, d \eta_j (\hat v) = \sum_{j\in\Asing} q_j \, \varphi (j, \hat r_j) \, .
\end{align*}

Recall that $\Qop \varphi = \Qinv \varphi + \Qsing \varphi$, where
\begin{align*}
\Qinv \varphi (j, \hat v) &= \sum_{ i\in\Ainv}  p_{ij} \,   \varphi(i, \hat A_i\, \hat v)\\
\Qsing \varphi (j, \hat v) &= \sum_{ i\in\Asing} p_{ij} \, \varphi(i, \hat r_i) \, .
\end{align*}

Then
\begin{align*}
\int \Qop \varphi \, d \eta &= \Sinv (\Qop \varphi) + \Ssing (\Qop \varphi) \\
&= \Sinv (\Qinv \varphi) + \Sinv (\Qsing \varphi) + \Ssing (\Qinv \varphi) + \Ssing (\Qsing \varphi) \, .
\end{align*}

We evaluate each of the four terms above separately.
\begin{align*}
\Sinv (\Qinv \varphi) &=  \sum_{j\in\Ainv} \, \sum_{i\in\Ainv} \,  \sum_{s\in\Asing} q_s \, \sum_{n=1}^\infty \sum_{\uom\in\Bscr_n (s, j)} p (\uom) \, p_{i j} \, \varphi (i, \hat A_i \hat A^n(\uomega) \hat r_s)\\
&=  \sum_{i\in\Ainv} \, \sum_{s\in\Asing} q_s \, \sum_{n=1}^\infty \sum_{\uom'\in\Bscr_{n+1} (s, i)} p (\uom') \, \varphi (i, \hat A^n(\uomega') \hat r_s)\\
&=  \sum_{i\in\Ainv} \, \sum_{s\in\Asing} q_s \, \sum_{n=2}^\infty \sum_{\uom\in\Bscr_{n} (s, i)} p (\uom) \, \varphi (i, \hat A^n(\uomega) \hat r_s) \, .
\end{align*}
The first line turned into the second by adding the letter $i$ to the end of the word $\uom\in\Bscr_n (s, j)$ with $j\in\Ainv$, so that $\uom' := (\uom, i) \in \Bscr_{n+1} (s, i)$; note that $p (\uom') = p_{i j} \, p(\uom)$. 
Similarly,
\begin{align*}
\Sinv (\Qsing \varphi) &=  \sum_{j\in\Ainv} \, \sum_{i\in\Asing} \,  \sum_{s\in\Asing} q_s \, \sum_{n=1}^\infty \sum_{\uom\in\Bscr_n (s, j)} p (\uom) \, p_{i j} \, \varphi (i, \hat r_i)\\
&=  \sum_{i\in\Asing} \, \sum_{s\in\Asing} q_s \, \sum_{n=1}^\infty \sum_{\uom'\in\Bscr_{n+1} (s, i)} p (\uom') \, \varphi (i,  \hat r_i)\\
&=  \sum_{i\in\Asing} \, \sum_{s\in\Asing} q_s \, \sum_{n=2}^\infty \sum_{\uom\in\Bscr_{n} (s, i)} p (\uom) \, \varphi (i,  \hat r_i) \, .
\end{align*}

Then
\begin{align*}
\Ssing (\Qinv \varphi) &=  \sum_{j\in\Asing} q_j \sum_{i\in\Ainv} p_{i j} \, \varphi (i,  \hat A_i \hat r_j)\\
&=  \sum_{i\in\Ainv} \, \sum_{s\in\Asing} q_s \, \sum_{\uom\in\Bscr_{1} (s, i)} p (\uom) \, \varphi (i, \hat A^1 (\uom)  \hat r_s)
\end{align*}
where in passing from the second to the last line we changed the index $j$ for $s$ and used the fact that $(s, i)$ is the only word $\uom\in\Bscr_1 (s, i)$ and in this case $p(\uom)=p_{i s}$ and $A^1 (\uom)=A_i$.

Finally,
\begin{align*}
\Ssing (\Qsing \varphi) &=  \sum_{j\in\Asing} q_j \sum_{i\in\Asing} p_{i j} \, \varphi (i,  \hat r_i)\\
&=  \sum_{i\in\Asing} \, \sum_{s\in\Asing} q_s \, \sum_{\uom\in\Bscr_{1} (s, i)} p (\uom) \, \varphi (i,  \hat r_i) \, .
\end{align*}

Note that 
$\Sinv (\Qinv \varphi) + \Ssing (\Qinv \varphi) = \Sinv (\varphi)$
and using again~\eqref{q formula},
\begin{align*}
\Sinv (\Qsing \varphi) + \Ssing (\Qsing \varphi) &=  \sum_{i\in\Asing} \Big(\sum_{s\in\Asing} q_s  \sum_{n=1}^\infty \sum_{\uom\in\Bscr_{n} (s, i)} p (\uom)\Big)  \varphi (i,  \hat r_i) \\
&=  \sum_{i\in\Asing} q_i \,  \varphi (i,  \hat r_i) = \Ssing (\varphi) ,
\end{align*}
which concludes the proof.
\end{proof}

\begin{theorem} \label{Uniform Ergodicity} If the cocycle $(\uA, P, q)$ has both singular and invertible components, then the corresponding Markov operator $\Qop$ is uniformly ergodic, i.e. there exist constants $C<\infty$ and $a>0$ such that
$$  \norm{  \Qop^n\varphi - \int\varphi\, d\eta }_\infty \leq C\, e^{-a\,n}\,\norm{\varphi}_\infty $$
for all $n\in\N$ and  
$\varphi\in L^\infty(\Ascr\times\Pp^1)$, where $\eta$ is any $\Qop$-stationary measure.
\end{theorem}

\begin{proof} Note that $L^\infty (\Ascr) \simeq \R^k$, seen as the subspace of observables $\varphi$ that do not depend on the projective variable, is invariant under the Markov operator $\Qop$; when restricted to this subspace, $\Qop$ becomes the operator $u \mapsto u \, P$ induced by the left stochastic matrix $P$. Recall also that $\eta$ projects down to $q$, the stationary probability vector of $P$. Since $P$ is primitive, for some constants $C<\infty$ and $\sigma \in (0, 1)$ (these constants may change a few times throughout the argument) 
\begin{equation}\label{eq 999}
\norm{u \, P^n - u \cdot q}_\infty \le C \sigma^n \, \norm{u}_\infty
\end{equation}
for all $n\in\N$ and $u \in \R^k \simeq L^\infty (\Ascr)$.

We show that for all $n\in\N$ there is a linear operator $T_n \colon L^\infty(\Ascr\times\Pp^1) \to L^\infty(\Ascr)$ such that for all $\varphi\in L^\infty(\Ascr\times\Pp^1)$,
\begin{equation}\label{eq1000}
\Qop^n \varphi = \Qinv^n \varphi + T_n \varphi .
\end{equation}
In other words, $\Qop^n \varphi - \Qinv^n \varphi$ does not depend on the projective variable. 

Indeed, when $n=1$, $\Qop \varphi = \Qinv \varphi + \Qsing \varphi$, where $\Qsing \varphi (j, \hat v) = \sum_{i\in\Asing} \varphi (i, \hat r_i) \, p_{i j} =: T_1 \varphi (j)$ does not depend on $\hat v$. By induction, if we assume that for all observables $\varphi$ we can write $\Qop^n \varphi = \Qinv^n \varphi + T_n \varphi$ with $T_n \varphi$ not depending on the projective variable, then
\begin{align*}
\Qop^{n+1} \varphi &= \Qop (\Qop^n \varphi) = \Qop (\Qinv^n \varphi + T_n \varphi) 
=  \Qop (\Qinv^n \varphi) + \Qop (T_n \varphi) \\
&= \Qinv (\Qinv^n \varphi) + \Qsing (\Qinv^n \varphi) +  (T_n \varphi) \, P
= \Qinv^{n+1} \varphi  + T_{n+1} \varphi ,
\end{align*}
where 
\begin{equation}\label{eq1002}
T_{n+1} \varphi :=  \Qsing (\Qinv^n \varphi) +  (T_n \varphi) \, P
\end{equation}
does not depend on the projective variable, thus establishing~\eqref{eq1000}.

\smallskip

We show that $\Qinv^n$ eventually contracts, so $\Qinv^n \varphi$ converges exponentially fast to $0$, while  $T_n \varphi$ converges exponentially fast to a constant (in the sup norm), which immediately imply the conclusion of the theorem. 

\smallskip

Since the stochastic matrix $P$ is primitive, there is $N\ge1$ such that $P^N_{i j} > 0$ for all $i, j \in \Ascr$. In particular, since $P^N$ is also stochastic,
$$\sigma_0 := \max_{j\in\Ascr} \, \sum_{i\in\Ainv} \, P^N_{i j} < 1 .$$

We show that $\Qinv^N$ is a contraction on $L^\infty(\Ascr\times\Pp^1)$. Indeed, given any observable $\varphi$ and $(j, \hat v) \in \Ascr \times \Pp^1$, clearly
$$\Qinv^N \varphi (j, \hat v) = \sum_{i\in\Ainv} \, \sum_{\uom\in\Bscr_N (j, i)} \varphi (i, A^N (\uom) \hat v) \, p (\uom) \, .$$

Then
\begin{align*}
\norm{\Qinv^N \varphi}_\infty &\le \norm{\varphi}_\infty \, \max_{j\in\Ascr} \, \sum_{i\in\Ainv} \, \sum_{\uom\in\Bscr_N (j, i)} p (\uom) \\
&= \norm{\varphi}_\infty \, \max_{j\in\Ascr} \, \sum_{i\in\Ainv} \, \sum_{\uom'\in\Ainv^{N-1}} p ( j \uom' i)\\
&\le \norm{\varphi}_\infty \, \max_{j\in\Ascr} \, \sum_{i\in\Ainv} \, \sum_{\uom'\in\Ascr^{N-1}} p ( j \uom' i)\\
&=  \norm{\varphi}_\infty \, \max_{j\in\Ascr} \, \sum_{i\in\Ainv} \, P^N_{i j} = \sigma_0 \,  \norm{\varphi}_\infty \, .
\end{align*}

Given any integer $n$, we write $n = m N + r$ with $0\le r < N$ and use the fact that the operator norm of $\Qinv$ is equal to $1$ to  conclude that there are $C<\infty$, $\sigma \in (0, 1)$ so that for all $n\in\N$ and $\varphi\in L^\infty(\Ascr\times\Pp^1)$,
\begin{equation}\label{eq2001}
\norm{\Qinv^n \varphi}_\infty \le C \sigma^n \, \norm{\varphi}_\infty .
\end{equation}

From~\eqref{eq1002} and~\eqref{eq2001} we get that for all $n\in\N$,
$$\norm{ T_{n+1} \varphi - (T_n \varphi) \, P}_\infty = \norm{ \Qsing (\Qinv^n \varphi) }_\infty \le 
\norm{\Qinv^n \varphi}_\infty \le C \sigma^n \, \norm{\varphi}_\infty ,$$
\begin{align*}
\norm{T_{n+2} \varphi  - (T_n \varphi) \, P^2}_\infty &\le \norm{ T_{n+2} \varphi - (T_{n+1} \varphi) \, P}_\infty +
\norm{ (T_{n+1} \varphi  - (T_{n} \varphi) P) \, P}_\infty \\
&\le C (\sigma^n + \sigma^{n+1})\, \norm{\varphi}_\infty ,
\end{align*}
and proceeding inductively, for all $n, k \in \N$, 
$$\norm{ T_{n+k} \varphi - (T_n \varphi) \, P^k}_\infty \le C' \sigma^n \, \norm{\varphi}_\infty .$$

In particular,
\begin{equation}\label{eq1005}
\norm{ T_{2 n} \varphi - (T_n \varphi) \, P^n}_\infty \le C' \sigma^n \, \norm{\varphi}_\infty .
\end{equation}

By~\eqref{eq 999},
\begin{equation}\label{eq1006}
\norm{(T_n \varphi) \, P^n - (T_n \varphi) \cdot q}_\infty \le C \sigma^n \, \norm{(T_n \varphi)}_\infty \le 2 C \sigma^n \, \norm{\varphi}_\infty \, .
\end{equation}

We have $\Qop^n \varphi = \Qinv^n \varphi  + T_n \varphi $. Since the measure $\eta$ is $\Qop$-stationary and it projects down to $q$, it follows that
\begin{align}
\abs{ \int \varphi d \eta -  (T_n \varphi) \cdot q} &= \abs{ \int \Qop^n \varphi \, d \eta -   \int T_n \varphi \, d\eta } \notag\\
& = \abs{ \int \Qinv^n \varphi \, d \eta} \le \norm{\Qinv^n \varphi}_\infty \le C \sigma^n \, \norm{\varphi}_\infty . \label{eq1007}
\end{align}

Combining~\eqref{eq1005},~\eqref{eq1006} and~\eqref{eq1007}, we conclude that
\begin{equation*}\label{eq1010}
\norm{ T_{2 n} \varphi -  \int \varphi d \eta}_\infty \le C'' \sigma^n \, \norm{\varphi}_\infty ,
\end{equation*}
which, together with~\eqref{eq2001} establishes our result (after choosing the constants $C, a$ appropriately).
\end{proof}

\begin{corollary}
If the cocycle $(\uA, P, q)$ has both singular and invertible components then it admits a unique stationary measure.
\end{corollary}

\begin{proof}
If $\eta_1$ and $\eta_2$ are $\Qop$-stationary, then Theorem~\ref{Uniform Ergodicity} above applies to each of them, so 
for any $\varphi\in L^\infty(\Ascr\times\Pp^1)$, $\Qop^n \varphi$ converges uniformly to $\int \varphi \, d \eta_1$ and to $\int \varphi \, d \eta_2$. Thus $\int \varphi \, d \eta_1 = \int \varphi \, d \eta_2$ for all observables $\varphi$, showing that $\eta_1 = \eta_2$.
\end{proof}

\begin{lemma} \label{prodNorm}
Given rank one matrices $B_1, B_2, \dots , B_n$ and a unit vector $r_0$,  we have that
$$
\norm{B_{n} \cdots B_{1}\, r_0} = \prod_{l=1}^n\norm{B_{l}\, r_{l-1}} \, ,
$$
where $ r_{l-1}$  is a unit vector in the range of $B_{l-1}$.
\end{lemma}

\begin{proof}
We write $B_1 r_0 = \la_1 r_1$, thus $\sabs{\la_1} = \norm{B_1 r_0}$. It follows that 
$B_2 B_1 r_0 = \la_1 B_2 r_1$ so 
$$\norm{B_2 B_1 r_0} = \sabs{\la_1} \norm{B_2 r_1} = \norm{B_2 r_1} \norm{B_1 r_0} \, .$$
From here,
$$
\norm{B_{n+1} B_n \big(B_{n-1} \cdots B_{1}\, r_0 \big) } = \norm{ B_{n+1} r_n} \norm{ B_n B_{n-1} \cdots B_{1}\, r_0 }$$ 
and the conclusion follows by induction.
\end{proof}

Next we derive a closed-form expression for the Lyapunov exponent which, together with the explicit formula for the stationary measure in Proposition~\ref{prop stationary measure} will eventually lead to a Furstenberg-type formula. 

\begin{lemma} \label{propL1}
If the random cocycle $(\uA, P, q)$ has both singular and invertible components then
$$
L_1(\uA) = \sum _{i \in \Asing} \sum _{j \in \Asing} \sum _{n=1}^\infty \sum _{\uom\in \Bscr_{n}(i, j)} q_i \, p(\uom) \, \log \norm{A^n(\uom) r_i} \, .
$$
\end{lemma}

\begin{proof}
Consider the cylinders $C_i := [0; i]$ with $i \in \Asing$ and their union $C:= \bigcup_{i \in \Asing} C_i$, the set of all (bi-infinite) words with a singular symbol at the zeroth position. Then $q_0 := \mu(C) = \sum_{i\in \Asing} q_i$.

Moreover, let $\displaystyle C_{i, j}:= \bigcup_{n=1}^\infty \, \bigcup_{\uom\in \Bscr_{n}(i, j)} [0; \uom]$. 

By the mixing of the Markov shift, these sets give rise to the  following partitions modulo a zero measure set (mod $0$)
\begin{align*}
C_i &= \bigcup_{j \in \Asing}C_{i, j} =\bigcup_{j \in \Asing}\bigcup_{n=1}^\infty \, \bigcup_{\uom\in \Bscr_{n}(i, j)} [0; \uom] \;  \text{ and } \\
C & = \bigcup_{j \in \Asing}\bigcup_{i \in \Asing}C_{i, j} = \bigcup_{i \in \Asing}\bigcup_{j \in \Asing} \, \bigcup_{n=1}^\infty \, \bigcup_{\uom\in \Bscr_{n}(i, j)} [0; \uom] \, .
\end{align*}

Let $g \colon C \to C$ be the first return map to the cylinder $C$, given by $g(\omega)=\sigma^{\tau(\omega)}(\omega)$, where $\tau(\omega)=\min\{k \ge 1 \colon \sigma^k (\om) \in C \}$. The map $g$ preserves the induced measure $\bar \mu_C=\mu(C)^{-1}\, \mu|_C = \frac{1}{q_0} \, \mu|_C $.

We define the induced cocycle $F_C \colon C\times \R^2 \to C \times \R^2$, given by $F_C(\omega,v):=(g(\omega), \Ccal (\om) \,v)$, where $\Ccal (\om) := A^{\tau(\omega)}(\omega)$. 

By \cite[Proposition 4.18 and Exercise 4.8]{Viana2014} (it is easy to verify that these statements also hold for $\Mat_2 (\R)$-valued fiber maps) its Lyapunov exponent is related to that of the original cocycle $(\uA, P, q)$ via the expression $L_1(F_C)=\frac{1}{q_0} L_1(\uA)$; thus it is enough to compute $L_1(F_C)$. 

The induced cocycle $F_C$ leaves invariant the $1$-dimensional sub-bundle $X \ni \om \mapsto R(\omega):= \mathrm{Range}( A_{\omega_0})$.
Then using Oseledets' theorem, its first Lyapunov exponent is the rate of exponential growth of the fiber iterates of $F_C$ along this sub-bundle. Thus for $\mu$-a.e. $\om \in C$ and for a unit vector $r_0 \in R (\om) = \mathrm{Range}( A_{\omega_0})$,
\begin{equation}\label{eq 100}
L_1(F_C) = \lim_{n \to \infty} \frac{1}{n}\log \norm{\mathcal{C}^n(\omega) \, r_0}
= \lim_{n \to \infty} \frac{1}{n}\log \ \norm{  \prod_{l=1}^{n} \Ccal (g^l\omega) \, r_0 } \, .
\end{equation}

Given $\om \in C$, let $0 = k_1 < k_2 < \cdots $ denote all future entries to the singular part of the alphabet, that is, $k \in \N$ is such that $\om_k \in \Asing$ if and only if $k = k_l$ for some $l\in\N$. Then clearly $g^l (\om) = \sigma^{k_l} (\om)$ (whose zeroth entry is $\om_{k_l}$) and $\tau (g^l \om) = k_{l+1} - k_l$ for all $l\in\N$. Moreover, 
$$B_l := \Ccal (g^l\om) = A^{\tau (g^l \om)} (g^l \om) = A^{k_{l+1} - k_l} (\sigma^{k_l} \om)
= A_{\om_{k_{l+1}}} \cdots  A_{\om_{k_{l}+1}} \, ,$$
which is a rank one matrix whose range $r_l :=  \mathrm{Range} (B_l) = r_{\om_{k_{l+1}}} $.

By Lemma~\ref{prodNorm},
\begin{align}
\norm{\Ccal^n(\omega) \, r_0} &=  \norm{\prod_{l=n}^{1} \Ccal (g^l\omega) \, r_0} 
 =  \norm{\prod_{l=n}^{1} B_l \, r_0} =  \prod_{l=1}^n\norm{B_{l}\, r_{l-1}} \notag\\
& =  \prod_{l=1}^n \norm{ \Ccal (g^l \om) \, r_{\om_{k_{l}}} } \, . \label{eq 99}
\end{align}

Consider the observable $\varphi \colon C \to \R$, 
$$\varphi (\om) := \log \norm{\Ccal (\om) \, r_{\om_0} } = \log \norm{A^{\tau(\om)} (\om) \, r_{\om_0} } .$$
By~\eqref{eq 100},~\eqref{eq 99} and Birkhoff's ergodic theorem,
\begin{align*}
L_1(F_C) &=  \lim_{n \to \infty} \frac{1}{n} \sum_{l=1}^{n} \log \norm{ \Ccal (g^l \om) \, r_{\om_{k_{l}}} } 
=  \lim_{n \to \infty} \frac{1}{n} \sum_{l=1}^{n} \varphi (g^l \om) \\
&= \int_C  \varphi (\om) \, d \bar{\mu}|_C (\om)  = \frac{1}{q_0} \int_C \log \norm{\Ccal (\om) \, r_{\om_0} } \, d \mu (\om) \, .
\end{align*}

Note that on each given cylinder $[0; \uom]$ with $\uom\in \Bscr_{n}(i, j)$ for some $i, j \in \Asing$ and $n\ge 1$, the first return map $\tau$ is constant and equal to $n$, while the observable $\varphi$ is equal to
$\log \norm{A^n (\om)  \, r_{\om_0}} = \log \norm{A^n (\uom) \, r_i}$, thus it is constant. 
Moreover, the Markov measure of this cylinder is $\mu [0; \uom] = q_i \, p (\uom)$. Since these cylinders partition (mod $0$) the set $C$, we conclude that
$$\int_C \varphi (\om) \, d \mu (\om) = \sum_{i, j \in\Asing} \, \sum_{n=1}^\infty \, \sum_{\uom \in \Bscr_n (i, j)} q_i \, p (\uom)
\,   \log \norm{A^n (\uom) \, r_i} \, ,$$
which completes the proof of the lemma.
\end{proof}

\smallskip

Consider the following observable $\Psi \colon \Ascr\times \Pp^1 \to [- \infty, \infty)$,
$$
\Psi (j, \hat v) = \sum_{i \in \Ascr} p_{i j}\, \log \frac{\norm{A_i\, v}}{\norm{v}} .
$$

We are ready to phrase and prove a Furstenberg-type formula relating the first Lyapunov exponent of a random cocycle $(\uA, P, q)$ to the integral of the observable $\Psi$ with respect  to the stationary measure $\eta$ of its associated Markov operator.

\begin{theorem}[Furstenberg's Formula] \label{Furstenberg's Formula} 
If the random cocycle $(\uA, P, q)$ has both singular and invertible components then
$$
L_1(\uA) = \int \Psi \; d\eta .
$$
\end{theorem}

\begin{proof} 
Using the explicit formula of the stationary measure $\eta$ derived in Proposition~\ref{prop stationary measure},
\begin{align*}
\int  \Psi \ d\eta &= \sum_{l\in\Ascr}  \, \sum_{i\in \Ascr} p_{il}  \sum_{s\in \Asing}q_s\sum_{n=1}^{\infty}\sum_{\uom\in \Bscr_n (s, l)} p(\uomega)\log\frac{\norm{A_i\, A^n(\uomega)\,r_s}}{\norm{A^n(\uomega)\,r_s}} .
\end{align*}

For $\uom\in\Bscr_n (s, l)$ with $l\in\Asing$,  $\hat A^n(\uomega)\, \hat r_s = \hat A_l \cdots \hat A_{\om_1} \, \hat r_s = \hat r_l$. The expression above can then be split into
\begin{align*}
&\sum_{l\in\Asing} \sum_{i\in\Ascr}p_{il}\sum_{s\in \Asing}q_s\sum_{n=1}^{\infty}\sum_{\uomega\in\Bscr_n(s, l)} p(\uomega)\log\norm{A_i\,r_l}\\
&+ \sum_{l\in\Ainv}\sum_{i\in\Ascr}p_{il}\sum_{s\in \Asing}q_s\sum_{n=1}^{\infty}\sum_{\uomega\in\Bscr_n(s,l)} p(\uomega)\log\frac{\norm{A_i\, A^n(\uomega)\, r_s}}{\norm{A^n(\uomega)\,r_s}} .
\end{align*}

By~\eqref{q formula} we have $\sum_{s\in \Asing}q_s\sum_{n=1}^{\infty}\sum_{\uomega\in\Bscr_n(s, l)} p(\uomega) = q_l $ which, when used in the first term above, turns the last expressions into

\begin{align*}
&\sum_{l\in\Asing}\sum_{i\in\Ascr}p_{il}\,q_l\,\log\norm{A_i\,r_l}\\
&+\sum_{l\in\Ainv}\sum_{i\in\Ascr}p_{il}\sum_{s\in \Asing}q_s\sum_{n=1}^{\infty}\sum_{\uomega\in\Bscr_n(s, l)} p(\uomega)\log\norm{A_i\,A^n(\uomega)\,r_s}\\
&-\sum_{l\in\Ainv}\sum_{i\in\Ascr}p_{il}\sum_{s\in \Asing}q_s\sum_{n=1}^{\infty}\sum_{\uomega\in\Bscr_n(s, l)} p(\uomega)\log\norm{A^n(\uomega)\,r_s} \\
\\
&= \sum_{l\in\Asing}\sum_{i\in\Ascr}p_{il}\,q_l\,\log\norm{A_i\,r_l}\\
&+\sum_{i\in\Ascr}\sum_{s\in\Asing} q_s\sum_{n=2}^{\infty}\sum_{\uomega' \in \Bscr_n(s, i)}p(\uomega')\log\norm{A^n(\uomega')\, r_s} \\
&-\sum_{l\in\Ainv}\sum_{s\in \Asing}q_s\sum_{n=1}^{\infty}\sum_{\uomega\in\Bscr_n(s, l)} p(\uomega)\log\norm{A^n(\uomega)\,r_s} \, .
\end{align*}

The third line above turned into the sixth because $\sum_{i\in\Ascr} p_{il} = 1$. The second line turned into the fifth by adding the letter $i$ to the end of the word $\uom\in\Bscr_n (s, l)$ with $l\in\Ainv$, so that $\uom' := (\uom, i) \in \Bscr_{n+1} (s, i)$; note that $p (\uom') = p_{i l} \, p(\uom)$. 

\smallskip

The last two lines above cancel each other out partially (note the slight difference in the ranges of the index $n$ and the fact that the final letter in the words appearing in the penultimate term ranges over the whole alphabet, while in the ultimate  term it ranges only over the invertible part). 

The whole expression is then equal to
 \begin{align*}
& \sum_{l\in\Asing}\sum_{i\in\Ascr}p_{il}\,q_l\,\log\norm{A_i\,r_l}\\
&+\sum_{i\in\Asing}\sum_{s\in\Asing} q_s\sum_{n=2}^{\infty}\sum_{\uomega' \in \Bscr_n(s, i)}p(\uomega')\log\norm{A^n(\uomega')\, r_s} \\
&-\sum_{l\in\Ainv}\sum_{s\in \Asing}q_s \sum_{\uomega\in\Bscr_1(s, l)} p(\uomega)\log\norm{A^1(\uomega)\,r_s} \\
\\
&= \sum_{l\in\Ascr}\sum_{s\in\Asing} q_s \, p_{l s}\, \log\norm{A_l\,r_s}\\
&+\sum_{l\in\Asing}\sum_{s\in\Asing} q_s\sum_{n=2}^{\infty}\sum_{\uomega \in \Bscr_n(s, l)}p(\uomega)\log\norm{A^n(\uomega)\, r_s} \\
&-\sum_{l\in\Ainv}\sum_{s\in \Asing}q_s \, p_{l s} \log\norm{A_l \,r_s} \\
\\
&= \sum_{l\in\Asing}\sum_{s\in\Asing} q_s \, p_{l s}\, \log\norm{A_l\,r_s}\\
&+\sum_{l\in\Asing}\sum_{s\in\Asing} q_s\sum_{n=2}^{\infty}\sum_{\uomega \in \Bscr_n(s, l)}p(\uomega)\log\norm{A^n(\uomega)\, r_s} \\
& = \sum_{l\in\Asing}\sum_{s\in\Asing} q_s\sum_{n=1}^{\infty}\sum_{\uomega \in \Bscr_n(s, l)}p(\uomega)\log\norm{A^n(\uomega)\, r_s} \, .
\end{align*} 

The first line turned into the fourth after re-denoting the index $i$ by $l$ and the index $l$ by $s$; the final equality follows since $(s, l)$ is the only word $\uom\in\Bscr_1 (s, l)$ and in this case $p(\uom)=p_{l s}$ and $A^1 (\uom)=A_l$.
The expression in the last line is exactly the formula obtained in Lemma~\ref{propL1} for $L_1 (\uA)$ so the proof of Furstenberg's formula is complete.
\end{proof}

\begin{remark}
We note that the classical Furstenberg's formula (i.e.~\cite[Theorem 8.5]{Fur63}) for random cocycles in $\GL_m (\R)$, or other more recent versions thereof, are obviously not applicable to this singular setting. Moreover, the probabilistic approach used by Furstenberg and Kifer to establish such results (see e.g.~\cite[Theorem 1.4]{FK83}) is not immediately applicable either, since the observable $\Psi$ is not continuous, not even bounded on $\Ascr \times \Pp^1$. That is why we had to employ an ad-hoc argument which uses the explicit formula of the stationary measure. 
\end{remark}

\begin{remark}
The uniform ergodicity established in Theorem~\ref{Uniform Ergodicity} is the strongest form  of mixing of a transition operator, see~\cite[Section 16.2]{meyn_tweedie_glynn_2009} for a discussion of this and other types of mixing. It is equivalent to the quasi-compactness of the operator $\Qop$ on the {\em entire} space  $L^\infty(\Ascr\times\Pp^1)$. 

We note that in the classical case of $\GL_m (\R)$-valued cocycles satisfying an irreducibility assumption, the corresponding Markov operator was shown to be quasi-compact on the smaller subspace of observables that are H\"older continuous in the projective variable, endowed with the stronger H\"older norm, see for instance~\cite[Proposition 5.22]{DK-book}. This was then used to establish statistical properties (and the H\"older continuity) of the corresponding Lyapunov exponent. 

In the setting of this paper, however, the observable $\Psi$ appearing in Furstenberg's formula does {\em not} belong to the space $L^\infty(\Ascr\times\Pp^1)$, since it is not bounded. Thus the mixing property (i.e. the convergence $\Qop^n \varphi \to \int \varphi \, d \eta$) does not apply to $\Psi$, which is what makes the establishing of statistical properties in the next section an extremely technical matter.
\end{remark}

\begin{remark}
The stochastic dynamical system over the alphabet $\Ascr$, governed by the stochastic matrix $P$, induces a singular (partial) random walk on $\Ascr\times\Pp^1$, where some paths terminate upon reaching the kernel of the subsequent transition matrix. In fact, termination occurs with probability $1$ whenever the cocycle $\uA$  contains null words.

Nevertheless, the transition operator $\Qop \colon L^\infty(\Ascr\times\Pp^1)\to L^\infty(\Ascr\times\Pp^1)$, introduced at the beginning of this section to reflect the projective action of the Markov kernel, is desingularized. Specifically, if $i \in \Asing$, the term $\varphi(i, \hat A_i\, \hat v)$, which would  otherwise exhibit a discontinuity at the kernel of $A_i$, is replaced in the definition of the operator by the constant function $\varphi(i, \hat r_i)$.

This desingularization explains why the Markov operator $\Qop$ always admits stationary measures (in the sense of Definition~\ref{def stat measure}) and it is uniformly ergodic (Theorem~\ref{Uniform Ergodicity}). It also clarifies why Furstenberg's formula in Theorem~\ref{Furstenberg's Formula} remains valid for cocycles with both singular and invertible components.
\end{remark}

 \section{Statistical Properties}\label{StatProp}
 For simplicity of exposition, all statements in this section will be proven in the Bernoulli setting, but it will be clear that the arguments easily extend to the mixing Markov shift setting. 
The main difference between the two models, when it comes to these results, is the form of the transition operator and the explicit formula of its stationary measure, which is more cumbersome in the Markov case (as seen in  Proposition~\ref{prop stationary measure}). 

Let us summarize the relevant concepts in the Bernoulli setting.
Given the alphabet $\Ascr = \{1, \ldots, k\}$ and a partition $\Ascr = \Asing \sqcup \Ainv$ into two nonempty sets, consider the set $\cocycles \subset \Mscr$ of all $k$-tuples $\uA = (A_1, \ldots, A_k) \in \Mat_2(\R)^k$ such that  $\rank A_i = 1$ if $i \in \Asing$, $\det A_j > 0$ if $j\in\Ainv$ and, additionally, $\hat r_i \neq \hat k_j$ for all $i, j \in \Asing$. 

\smallskip

Recall that we identify such a $k$-tuple $\uA$ with the locally constant map $A \colon X \to \Mat_2 (\R)$, $A (\om) = A_{\om_0}$, where $\om = \{\om_n\}_{n\in \Z} \in X = \Ascr^\Z$. 
 

Given a probability vector $p = (p_1, \ldots, p_k)$ with $p_i > 0$ for all $i$, the pair $(\uA, p)$ determines the random linear cocycle 
$F \colon X \times \R^2 \to X \times \R^2$, $F (\om, v) = (\sigma \om, A (\om) v) = (\sigma \om, A_{\om_0} \, v)$
where $X$ is endowed with the product probability $\mu = p^\Z$.
As before, we also identify the cocycle $F$ with the tuple $\uA = (A_1, \ldots, A_k)$.

The Markov operator in this setting is $\Qop \colon L^\infty(\Pp^1) \to L^\infty(\Pp^1)$,
$$\Qop \varphi (\hat v) := \sum_{i\in\Ascr} \varphi (\hat A_i \hat v) \, p_i .$$

Given $n\ge 0$ and $s \in \Asing$, let $\Bscr_{n} (s)$ denote the set of $(n+1)$-tuples  $(s \uom)$ where $\uom=(\om_1, \ldots, \omega_{n}) \in \Ainv^n$.  
For such a word we write, as before, 
$A^{n}(s \uom):=A_{\omega_{n}}\, \ldots\, A_{\omega_2}\, A_{\omega_1}$  (with the convention that $A^0 (s \uom)$ is the identity matrix) and denote the quantity
 $p (\uom):= p_{\om_1} \cdots \, p_{\omega_{n}}$. 
 
 Let $q:= \sum_{i\in\Asing} p_i < 1$ and note that
\begin{align}\label{inca o suma}
\sum_{s \in \Asing} p_s \sum_{n=0}^\infty
	\sum_{\uomega \in \Ainv^n}  p(\uomega) = \sum_{s \in \Asing} p_s \sum_{n=0}^\infty (1-q)^n = q \, \frac 1 q = 1 .
\end{align}
  
With these notations we have the following explicit formula for the unique $\Qop$-stationary measure
$\eta\in\Prob (\Pp^1)$ (its derivation is quite straightforward compared to the Markov case, so we omit it):
$$
\eta =  \sum_{s \in \Asing} \sum_{n=0}^\infty
\sum_{\uom\in\Bscr_n(s)}  p(\uomega)\delta_{\hat A^n(s \uomega) \, \hat r_s} \, .
$$

We write $\Qop = \Qinv + \Qsing$, with $\Qinv \varphi (\hat v) := \sum_{i\in\Ainv} \varphi (\hat A_i \hat v) \, p_i$ and 
$\Qsing \varphi (\hat v) := \sum_{i\in\Asing} \varphi (\hat r_i) \, p_i$. Note that $\Qsing$ is a projection to the constant functions. This then easily implies by induction that for all $n\in\N$,
\begin{equation}\label{Qop Iterates}
\Qop^n = \Qinv^n + \Qsing\circ \Big(\sum_{i=0}^{n-1} \Qinv^i\Big)  \, .
\end{equation}
In particular, $\Qop^n - \Qinv^n$ takes any observable $\varphi$ to a constant function, which then implies the uniform ergodicity of the Markov operator.

That is, for all $n\in\N$ and  
$\varphi\in L^\infty(\Pp^1)$,
\begin{equation}\label{eq1001}
  \norm{  \Qop^n\varphi - \int\varphi\, d\eta }_\infty \leq C\, e^{-a\,n}\,\norm{\varphi}_\infty \, .
 \end{equation}

Finally, defining the observable $\psi \colon \Pp^1 \to [- \infty, \infty)$,
$$\psi (\hat v) := \sum_{i\in\Ascr} p_i \, \log \frac{\norm{A_i v}}{\norm{v}} \, ,$$
the Lyapunov exponent of the random cocycle $\uA$ is given by
\begin{equation}\label{F formula 10}
L_1 (\uA) = \int \psi (\hat v)  \, d \eta (\hat v) \, ,
\end{equation}
that is, a Furstenberg-type formula holds.

\medskip

The goal of this section is to establish an LDT estimate and a CLT for Lebesgue a.e.  cocycle $\uA \in \Mscr$. This will be obtained as a consequence of a stronger result essentially saying that if $t\mapsto\uA_t \in \cocycles$ is a one-parameter family of cocycles satisfying a positive winding condition, then an LDT estimate and a CLT hold for Lebesgue a.e. parameter $t$. By general results, see Cai, Duarte, Klein~\cite{CDK-paper3}, such limit laws hold as soon as an associated Markov operator satisfies an appropriate mixing condition. We will use a parameter elimination argument to show that for Lebesgue a.e. parameter $t$, the Markov operator corresponding to $\uA_t$ satisfies the desired condition. 
In order to apply the general results in~\cite{CDK-paper3} we need to associate to a given cocycle $\uA\in\cocycles$ a Markov operator on a slightly larger space, namely
$\bar{\Qop} \colon L^\infty(\Ascr \times \Pp^1) \to L^\infty(\Ascr \times \Pp^1)$ defined by
$$\bar{\Qop}\varphi(j,\hat{v}) = \sum_{i=1}^{k}  \varphi(i, \hat{A}_j \hat v) \, p_i \, .$$

Consider the projection $\pi \colon L^\infty(\Ascr \times \Pp^1) \to L^\infty(\Ascr)$ given by 
$$
\pi \varphi (\hat{v}) = \sum_{i=1}^{k}p_i  \varphi(i,\hat{v}) = \int_\Ascr \varphi (i, \hat v) \, d p (i) \, .
$$

Then the Markov operators $\bar \Qop$ and  $\Qop$ are related as follows.

\begin{lemma}\label{Diagrama}
The following diagram is commutative:
\[ \begin{tikzcd}
L^\infty(\Ascr \times \Pp^1) \arrow{r}{\bar{\Qop}} \arrow[swap]{d}{\pi} & L^\infty(\Ascr \times \Pp^1) \arrow{d}{\pi} \\%
L^\infty( \Pp^1) \arrow{r}{\Qop} & L^\infty(\Pp^1)
\end{tikzcd}
\]
That is, $\pi \circ \bar{\Qop}=\Qop \circ \pi$ and inductively, $\pi \circ \bar{\Qop}^n=\Qop^n \circ \pi$ for all $n\in\N$.
\end{lemma}

\begin{proof}
Let $\varphi \in L^\infty(\Ascr \times \Pp^1)$. Then,

$$\pi \circ \bar{Q} (\varphi) (\hat v) = \sum_{j=1}^k p_j \bar{Q} \varphi (j, \hat v) 
= \sum_{j=1}^k \sum_{i=1}^k p_j \, p_i \, \varphi (i, \hat A_j \hat v) ,$$
while
$$Q \circ \pi  (\varphi) (\hat v) = \sum_{j=1}^k p_j \pi \varphi (\hat A_j, \hat v) 
= \sum_{j=1}^k \sum_{i=1}^k p_j \, p_i \, \varphi (i, \hat A_j \hat v) $$
which establishes the claim.
\end{proof}

\begin{corollary}\label{cor strong mixing}
The measure $p \times \eta$ is $\bar{\Qop}$-stationary and $\bar{Q}$ is uniformly ergodic. More precisely, there exist constants $c$ and $C$ such that
$$
\norm{\bar{Q}^n \varphi - \int \varphi \; d\mu d\eta}_\infty \leq C e^{-cn}\norm{\varphi}_\infty \quad \quad \forall \; n \in \N, \; \forall \; \varphi \in L^\infty(\Ascr \times \Pp^1) .
$$
In particular $p \times \eta$ is the unique $\bar{\Qop}$-stationary measure.
\end{corollary}

\begin{proof} For any observable $\varphi \in L^\infty(\Ascr \times \Pp^1)$ we have
 \begin{align*}
 \iint \bar \Qop \varphi (i, \hat v) d p(i) d \eta (\hat v) &= \int \pi (\bar \Qop \varphi) (\hat v) d \eta (\hat v) 
 = \int \Qop (\pi \varphi) (\hat v) d \eta (\hat v) \\
 &= \int \pi \varphi (\hat v) d \eta (\hat v) = 
 \iint \varphi (i, \hat v) d p (i) d \eta (\hat v) ,
 \end{align*}
 which shows that $p \times \eta$ is $\bar{\Qop}$-stationary. 
Moreover,
\begin{align*}
\bar{\Qop}^{n+1} \varphi (j,\hat{v}) &= \bar{\Qop} \bar{\Qop}^n \varphi (j,\hat{v}) = \sum_{i=1}^k p_i \bar{\Qop}^n \varphi (i,\hat{A}_j \hat{v}) \\
&= \pi (\bar{\Qop}^n \varphi)  (\hat{A}_j \hat{v}) = \Qop^n(\pi \varphi)(\hat{A}_j \hat{v}).
\end{align*}

Since by~\eqref{eq1001} the operator $\Qop$ is strongly mixing, the last term converges uniformly and exponentially fast to $\int \pi \varphi \; d\eta = \iint \varphi \; d p d\eta$, thus showing the strong mixing of $\bar \Qop$.
\end{proof}

\smallskip

Let $\Lambda \subset \R$ be a compact interval (normalized to have length $1$)  and let $A\colon \Lambda \to \cocycles$ be a smooth map. We think of this map as a one-parameter family of random cocycles so we use the subscript notation $t \mapsto A_t$. For its components we write $A_t (i)$ instead of $(A_t)_i$, while the fiber iterates are denoted by
$$A_t^n (\om) = A_t (\om_{n}) \cdots A_t (\om_1) \, .$$

For every parameter $t\in\Lambda$, denote by $\Qop_t$ and $\bar \Qop_t$ the Markov operators corresponding to the cocycle $A_t$ (defined as above). Moreover, let $\eta_t \in \Prob (\Pp^1)$ be the unique $\Qop_t$-stationary measure. 

All the results proven above for a given cocycle $\uA\in\Mscr$, namely the explicit formula of the stationary measure, Furstenberg's formula and the strong mixing of the Markov operator, apply to $A_t$ for all $t\in\Lambda$; since the probability vector $p = (p_1, \ldots, p_k)$ and the singular/invertible symbols do not change, all the mixing parameters are uniform in $t$. Finally, recall that the stationary measure is given by:
$$
\eta_t =  \sum_{s \in \Asing} p_s \, \sum_{n=0}^\infty
\sum_{\uom \in \Ainv^n}  p(\uomega) \delta_{\hat A_t^n(s\uomega) \, \hat r_s} \, .
$$   

\subsection{Assumptions}\label{assumptions} We assume that the family of cocycles $t \mapsto \uA_t$ is positively winding (see~\cite[Definition 4.1]{DDGK-paper1}) and that its singular components are constant. More precisely, we impose the following conditions on the smooth map   $\Lambda \ni t \mapsto \uA_t \in \cocycles$:
\begin{enumerate}

\item[(A1)] For all $i\in \Asing$, $A_t (i) \equiv A_i$.

\item[(A2)] There is $c_0 > 0$ such that for all $t \in \Lambda$, $j \in \Ainv$ and $\hat v \in \Pp^1$ we have
$$\frac{A_t (j)\, v \wedge \frac{d}{d t} \, A_t (j) v }{\norm{A_t (j) \, v}^2}   \ge c_0 \, .$$
\end{enumerate}

\begin{remark}By~\cite[Proposition 3.1]{BCDFK}, the quantity in item (A2) above, which we refer to as the rotation speed of the map $t\mapsto A_t$, can be characterized by
$$\frac{d}{d t} \, \hat A_t (j) \hat v = \frac{d}{d t} \, \frac{A_t (j) v}{\norm{A_t (j) v}} = \frac{A_t (j)\, v \wedge \frac{d}{d t} \, A_t (j) v }{\norm{A_t (j) \, v}^2} \, .$$
\end{remark} 

\begin{remark}
A more general version of the winding condition requires that the inequality  in assumption (A2) above holds for some iterate $\uA_t^{n_0}$ of the cocycle.  For simplicity we assume that $n_0=1$.
Moreover, if the assumption (A2) holds, then there exists $c_1 > 0$ that only depends on the family $A_t$, such that $\forall n \in \N$, $\forall \om \in \Ainv^n$, $\forall \hat{v} \in \Pp^1$, we have
\begin{equation} \label{Winding for An}
\frac{d}{dt} \frac{A_t^n(\om)v}{\norm{A_t^n(\om)v}} \geq c_1.
\end{equation}
In other words, if  $A_t$ is a family of positively winding cocycles, then so is $A_t^n$ for every $n \in \N$, with rotation speed uniformly (in $n$) bounded from below. For a proof of this statement see~\cite[Section 3.1]{BCDFK}, specifically 
Propositions 3.3 and 3.4.
\end{remark} 

\begin{example}
Given any tuple $\underline{A} = (A_1, \ldots, A_k) \in \cocycles$, the one-parameter family $\left[ - \pi , \pi \right] \ni t \mapsto A_t \in \Mat_2(\R)^k$,
$A_t (i) = A_i$ if $i \in \Asing$ and 
$A_t (j) = A_j \, R_{t} = A_j \, \begin{bmatrix} \cos  t & - \sin t \\ \sin  t & \cos t \end{bmatrix}$ if $j \in \Ainv$
satisfies the assumptions (A1) and (A2) above, passes through $\underline{A}$ and $\uA_t\in\cocycles$ for all $t\in\Lambda$.
\end{example}

\subsection{Some technical estimates}

In order to simplify the exposition in the estimates to follow, we will write $a \less b$ if there is some absolute or at least context universal constant $C < \infty$ such that $a \le C \, b$. Moreover, for an arc $\hat I \subset \Pp^1$, $m (\hat I)$ denotes its length, while $\Leb (E)$ stands for the Lebesgue measure of a subset $E$ of the real line.

\begin{lemma} \label{lemma finite arcs}
Given any $\varepsilon >0$, the set 
$$
\hat{I}_\varepsilon \colon = \Big\{ \hat{v} \in \Pp^1 \colon \norm{A_i \frac{v}{\norm{v}}} < \varepsilon \text{ for some } i \in \Asing \Big\}
$$
is a finite union of arcs with  $m(\hat{I}_\varepsilon) \less \varepsilon$.
\end{lemma}

\begin{proof}
Let $i \in \Asing$ and consider the normalized directions $r_i$ and $k_i$ of the range and kernel of $A_i$, respectively, such that  $\norm{k_i} = \norm{r_i}=1$. Given any vector $v \in \R^2$, we write $v = \alpha_1 r_i + \alpha_2k_i$. Then
$$
\norm{A_i \frac{v}{\norm{v}}} = \frac{| \alpha_1 |}{\norm{v}}.
$$
Thus $\norm{A_i \frac{v}{\norm{v}}} < \varepsilon$ in a small interval around each kernel $k_i$. Moreover, $\hat{I}_\varepsilon$ is a finite union of arcs of length $\asymp 2\varepsilon$ each. Hence $m(\hat{I}_\varepsilon) \lesssim \varepsilon$.
\end{proof}

\begin{lemma} \label{lemma Leb measure of t in I}
Given any arc $\hat I \subset \Pp^1$,  $n\in\N$, $\om \in \Ainv^n$ and $\hat v \in \Pp^1$ we have
$$
\Leb \left\{ t \in \Lambda \colon \hat{A}^n_t(\om) \hat{v} \in \hat{I} \right\} \le  \frac{m(\hat{I})}{c_1} .
$$
\end{lemma}

\begin{proof} Since the projective line $\Pp^1$ is one dimensional, we may regard the map $\Lambda \ni t \mapsto  \hat{A}^n_t(\om) \hat{v}$ as a real valued map, whose derivative (because of the winding condition) is bounded from below by $c_1$. The result then follows using basic calculus (the mean value theorem).
\end{proof}

Consider the maps $\varphi_t \colon \Ascr \times \Pp^1 \to [-\infty, \infty)$ and $\psi_t \colon \Pp^1 \to[-\infty, \infty)$ given by
$$
\varphi_t(i,\hat{v}) := \log \norm{A_t(i) \frac{v}{\norm{v}}} \, \text{ and } \, \psi_t(\hat{v})= \sum_{i \in \Ascr} p_i\varphi_t(i,\hat{v}) = \int_\Ascr \varphi_t (i, \hat v) d p(i) .
$$

Let $c=c(A)$ be a constant that satisfies

\begin{enumerate}
\item[(i)] $e^{-c } \le \norm{A_t(j)} \leq e^c \quad \forall t \in \Lambda \text{ and } \forall j \in \Ainv$.
\item[(ii)] $\norm{A_i} \leq e^c \quad \forall i \in \Asing$. 
\item[(iii)] $\norm{A_i r_l} \geq e^{-c} \quad \forall i, l \in \Asing $.
\end{enumerate}

Such a constant exists by the compactness of $\Lambda$ and the fact that $\uA \in \cocycles$, so $\hat k_i \neq \hat r_l$, thus  $A_i r_l \neq 0$ for all $i, l \in \Asing$. It follows that for all $j \in \Ainv$  and  $\hat{v} \in \Pp^1$ we have $\abs{  \varphi_t(j,\hat{v}) } \le c$. Moreover, the upper bound $\varphi_t(j,\hat{v}) \le c$  holds for every $j \in \Ascr$ and  $\hat{v} \in \Pp^1$.


%

\begin{lemma} \label{lem exp small measure}
There is  $C < \infty$ which depends on the constants $c, c_1$, so only on $\uA$, such that the following hold. \begin{enumerate}
\item[(i)] For all $n\in\N$, $\om \in \Ainv^n$,  $i \in \Asing$, $\hat v \in \Pp^1$ and $N\ge 0$ we have
$$
\Leb \{ t \in \Lambda \colon \varphi_t(i, \hat{A}^n_t(\om)\hat{v}) < - N \} \le C \, e^{-N} \, .
$$
Moreover, 
$$
\int_{\Lambda} \varphi_t^2(i,\hat{A}^n_t(\om)\hat{v}) \; dt \le C .
$$

\item[(ii)] For all  $n\in\N$,  $\om \in \Ascr^n$, $i, s \in \Asing$  and $N\ge 0$ we have
$$
\Leb \{ t \in \Lambda \colon \varphi_t (i,\hat{A}^n_t(\om)\hat{r}_s) < - N \} \le C \, e^{-N} \, .
$$
Moreover, 
$$
\int_{\Lambda} \varphi_t^2(i,\hat{A}_t^n(\om)\hat{r}_s) \; dt \le C .
$$

\item[(iii)] Furthermore, for all $i\in\Ascr$, $\om \in X$ and $n\in\N$,
$$\int_\Lambda \int_{\Pp^1} \varphi_t^2 (i, \hat A_t^n (\om) \hat v) \, d \eta_t (\hat v) \, d t \le C .$$

In particular,
$$\int_\Lambda \int_{\Pp^1} \psi_t^2 (\hat A_t^n (\om) \hat v) \, d \eta_t (\hat v) \, d t \le C .$$

%

\end{enumerate} 
\end{lemma}

\begin{proof}
(i) If $i\in\Asing$ then $\varphi_t (i, \hat w) < - N$ holds if and only if $\norm{A_i  \frac{w}{\norm{w}}} < e^{- N}$. By Lemma~\ref{lemma finite arcs} the set of such points $\hat w$ is a finite union of arcs with total measure of order $\varepsilon := e^{-N}$, and by Lemma~\ref{lemma Leb measure of t in I} the measure of the set of parameters $t$ for which $\hat{A}^n_t(\om)\hat{v}$ belongs to these arcs is $\less \varepsilon$. This proves the first statement in item (i).  

For the second statement, note that
\begin{align*}
&\int_{\Lambda} \varphi_t^2(i,\hat{A}^n_t(\om)\hat{v}) \; dt = \int_0^\infty \Leb \{ t \in \Lambda \colon \varphi_t^2 (i, \hat{A}^n_t(\om)\hat{v}) \geq x \} \; dx\\
&= \int_0^{2 c^2} \Leb \{ t \in \Lambda \colon \varphi_t^2 (i, \hat{A}^n_t(\om)\hat{v}) \geq x \} \; dx \\
&+ \int_{2 c^2}^{\infty} \Leb \{ t \in \Lambda \colon \varphi_t^2 (i, \hat{A}^n_t(\om)\hat{v}) \geq x \} \; dx \\
&\leq 2 c^2 + \int_{2 c^2}^{\infty} \Leb \{ t \in \Lambda \colon  \varphi_t (i, \hat{A}^n_t(\om)\hat{v})  \le - \sqrt{x}  \} \; dx \\
&\less c^2 + \int_{2 c^2}^{\infty} e^{-\sqrt{x}} \; dx \less 1 .
\end{align*}

The bound in the penultimate line above holds because $\Lambda$ has length one and, moreover,  the upper bound $ \varphi_t(i,\hat{w}) \leq c$ is valid for all $\hat w \in \Pp^1$, so that  $\varphi_t (i,\hat{A}^n_t(\om)\hat{v}) \le c < \sqrt{x}$ when $x\ge2c^2$. 

\smallskip

 (ii) Note that if $\om \in \Ainv^n$, then the statement follows from item (i). The difference here is that we consider any word $\om$, invertible or not, but begin with a singular vector, namely $r_s$. 
 
Let $\om = (\om_{0}, \ldots, \om_{n-1})  \notin \Ainv^n$. We split the argument into two cases: either $\om_{n-1} \in \Asing$ or else $\om_{n-1} \in \Ainv$ but there is $0 \leq j \leq n-2$ such that $\om_j \in \Asing$.

%
%
%
%
%

If $\om_{n-1} \in \Asing$, then 
$$
\hat{A}_t^n(\om)\hat{r_s} = \hat{A}_t(\om_{n-1})\hat{A}_t^{n-1}(\om)\hat{r_s} = \hat A_{\om_{n-1}} \, \hat{A}_t^{n-1}(\om)\hat{r_s} = \hat{r}_{\om_{n-1}}.
$$

Then by the choice of the constant $c$, for all parameters $t \in \Lambda$,
$$
\varphi_t (i, \hat{A}^n_t(\om)\hat{r_s}) = \log \norm{ \hat{A}_i \hat{r}_{\om_{n-1}} }  \geq -c ,
$$ 
so the set $\{ t \in \Lambda \colon \varphi_t (i,\hat{A}^n_t(\om)\hat{r}_s) < - N \}$ becomes empty for $N \ge c$ and the statement follows trivially.

If  $\om_{n-1} \in \Ainv$ and $\om_j \in \Asing$ for some index $0 \leq j \leq n-2$, let $n'$ be the largest such index and let $\om' :=  (\om_{n'+1}, \dots , \om_{n-1})\in \Ainv^{n-n'-1}$.
Then 
$$
\hat{A}^n_t(\om)\hat{r}_s = \hat{A}_t^{n-n'-1}(\om')\hat{A}_{\om_{n'}}\hat{A}^{n'}_t(\om)\hat{r}_s = \hat{A}^{n-n'-1}_t (\om') \hat{r}_{\om_{n'} }.
$$
Thus $\varphi_t(i, \hat{A}_t^n (\om)\hat{r}_s ) = \varphi_t(i, \hat{A}_t^{n-n'-1}(\om')\hat{r}_{\om_{n'}})$, with $\om' \in \Ainv^{n-n'-1}$, hence item (i) is applicable and the conclusion follows.

(iii) If $i \in \Ainv$ the statement is obvious since $\varphi_t^2 (i, \hat w) \le c^2$ for all $(t, \hat w) \in \Lambda \times \Pp^1$. Let us then fix $i\in\Asing$. Using the explicit formula of $\eta_t$ we have
\begin{align*}
&\int_\Lambda \int_{\Pp^1} \varphi_t^2 (i, \hat A_t^n (\om) \hat v) \, d \eta_t (\hat v) \, d t  \\
& =  \sum_{s \in \Asing}p_s \sum_{j=0}^\infty \sum_{\uom \in \Ainv^j} p(\uom) \int_{\Lambda} \varphi_t^2  (i, \hat A_t^n (\om) \, \hat{A}_t^j (s\uom)   \hat{r}_s) \, dt\\
& =  \sum_{s \in \Asing}p_s \sum_{j=0}^\infty \sum_{\uom \in \Ainv^j} p(\uom) \int_{\Lambda} \varphi_t^2 (i,\hat{A}^{n+j}_t(s\uom \om)\hat{r}_s ) \, dt\\
& \le \sum_{s \in \Asing}p_s \sum_{j=0}^\infty \sum_{\uom \in \Ainv^n} p(\uom) \, C = C,
\end{align*}
where the inequality in the last line follows from item (ii) above and the last equality comes from~\eqref{inca o suma}.
\end{proof}

%

Let $\tilde \eta \in \Prob (\Lambda \times \Pp^1)$ be the probability measure whose disintegration relative to the Lebesgue measure on $\Lambda$ is $\{\eta_t\}$, that is, the measure characterized by
$$\int \phi (t, \hat v) \, d \tilde\eta (t, \hat v) = \int_\Lambda \int_{\Pp^1}  \phi (t, \hat v) \, d \eta_t (\hat v) \, d t \quad \forall \phi \in L^\infty (\Lambda \times \Pp^1) \, .$$

\begin{lemma}\label{lema 10}
Given $i \in \Ascr$, $n\in \N$, $N > c$ and $\om \in X$, let  
\begin{align*}
P_{i, n, N}(\om) : = \left\{  (t, \hat{v}) \colon \varphi_t(i, \hat{A}^n_t(\om)\hat{v}) < -N  \right\} .
\end{align*}

Then $\tilde{\eta} (P_{i, n, N}(\om)) \le C   e^{-N}$.
\end{lemma}

\begin{proof} 
If $i\in\Ainv$ then $\varphi_t(i, \hat w) \ge - c > - N$ for all $\hat w \in \Pp^1$, so $P_{i, n, N}(\om)$ is empty. We then fix $i \in\Asing$ and note that 
\begin{align*}
 \tilde{\eta}(P_{i, n, N}(\om)) &= \int_{\Lambda} \int_{\Pp^1} \ind_{P_{i, n, N}(\om)} (t,\hat{v}) \; d\eta_t(\hat{v}) dt \\
 &= \sum_{s \in \Asing}p_s \sum_{j=0}^\infty \sum_{\uom \in \Ainv^j} p(\uom) \int_{\Lambda} \ind_{P_{i, n, N}(\om)} (t, \hat{A}_t^j (s\uom) \hat{r}_s) \; dt.
\end{align*}

Moreover,
\begin{align*}
(t, \hat{A}_t^j (s\uom) \hat{r}_s) \in P_{i, n, N}(\om) &\Longleftrightarrow \varphi_t(i, \hat{A}^n_t(\om)\hat{A}^j_t(s\uom)\hat{r}_s) <-N  \\
&\Longleftrightarrow \varphi_t(i,\hat{A}^{n+j}_t(s\uom \om)\hat{r}_s ) < -N.
\end{align*} 

By Lemma~\ref{lem exp small measure} item (ii), for all $j \in \N$,
\begin{align*}
\int_{\Lambda} \ind_{P_{i, n, N}(\om)} (t, \hat{A}^j_t(s\uom)\hat{r}_s) \; dt = \Leb\{ t \colon \varphi_t(i, \hat{A}^{n+j}_t(s\uom \om)\hat{r}_s <-N \} \lesssim e^{-N}
\end{align*}
and the conclusion follows from~\eqref{inca o suma}.
\end{proof}

Given any $N>c$, consider the truncation
$$\varphi_{t, N} := \max \left\{ \varphi_t, - N \right\} \, .$$
Note that $\varphi_{t, N} \in L^\infty(\Ascr\times\Pp^1)$ (a property that does not hold for $\varphi_t$) and $\norm{\varphi_{t, N}}_\infty \le N$. Moreover, 
$\varphi_t (i, \hat v) = \varphi_{t, N} (i, \hat v)$ if and only if $i\in\Ainv$ or else $i\in\Asing$ and $\varphi_t (i, \hat v) < - N$.

\begin{lemma} \label{lemma Estimates phi and psi} 
For all $i\in\Ascr$, $n\in\N$, $N > c$ and $\om\in X$ we have
$$
\int_{\Lambda} \int_{\Pp^1} \abs{ \varphi_t (i, \hat A_t^n (\om) \hat v) -  \varphi_{t, N} (i, \hat A_t^n (\om) \hat v)} \, d\eta_t (\hat{v}) \, dt \less e^{-N/3} \, .
$$
\end{lemma}

\begin{proof} The statement is obvious when $i\in\Ainv$ since in this case $\varphi_t (i, \hat w) = \varphi_{t, N} (i, \hat w)$ for all $\hat w \in \Pp^1$.

We  fix $i\in\Asing$ and recall that $\varphi_t (i, \hat A_t^n (\om) \hat v) \neq \varphi_{t, N} (i, \hat A_t^n (\om) \hat v)$ if and only if $\varphi_t (i,  \hat A_t^n (\om) \hat v) < - N$, that is, if and only if  $(t, \hat v) \in P_{i, n, N} (\om) =: P_N$, which by Lemma~\ref{lema 10} is a set of $\tilde \eta$-measure $\less e^{- N}$. Then 
\begin{align*}
& \int_{\Lambda} \int_{\Pp^1} \abs{ \varphi_t (i,  \hat A_t^n (\om) \hat v) -  \varphi_{t, N} (i,  \hat A_t^n (\om) \hat v)} \, d\eta_t (\hat{v}) \, dt \\
& = \int_{\Lambda \times \Pp^1} \abs{ \varphi_t (i,  \hat A_t^n (\om) \hat v) -  \varphi_{t, N} (i,  \hat A_t^n (\om) \hat v)} \, \ind_{P_N} (t, \hat v) \,  d \tilde \eta (t, \hat v)\\
& \le \norm{g (t, \hat v)}_{L^2 (\tilde \eta)} \, \,  \tilde \eta (P_N)^{1/2}
\end{align*}
where $g (t, \hat v) :=  \varphi_t (i,  \hat A_t^n (\om) \hat v) -  \varphi_{t, N} (i,  \hat A_t^n (\om) \hat v)$ and we used Cauchy-Schwarz in the last inequality above. 

Moreover, by Lemma~\ref{lem exp small measure} item (iii),
$$\norm{g (t, \hat v)}_{L^2 (\tilde \eta)} \le \norm{\varphi_t (i,  \hat A_t^n (\om) \hat v)}_{L^2 (\tilde \eta)} + \norm{\varphi_{t, N} (i,  \hat A_t^n (\om) \hat v)}_{L^2 (\tilde \eta)} \le \sqrt{C} + N .$$

Thus
$$\int_{\Lambda} \int_{\Pp^1} \abs{ \varphi_t (i,  \hat A_t^n (\om) \hat v) -  \varphi_{t, N} (i,  \hat A_t^n (\om) \hat v)} d\eta_t (\hat{v}) dt \less ( \sqrt{C} + N)  e^{-N/2} \less e^{-N/3} $$
which completes the proof.
\end{proof}

\smallskip

\subsection{Large deviations estimates}
In order to prove large deviations estimates for cocycles $\uA\in\cocycles$, we use the following abstract recent result of Cai, Duarte, Klein~\cite[Theorem 1.1.]{CDK-paper3}. 

Let $\{Z_n\}_{n\ge 0}$ be a Markov chain with values in a compact space $M$, and transition kernel $M\ni x \mapsto K_x \in \Prob (M)$. Let $Q \colon L^\infty (M) \to L^\infty (M)$ be the Markov operator corresponding to this kernel, namely $Q \phi (x) = \int_M \phi (y) d K_x (y)$. If the operator $Q$ is strongly mixing (in an appropriate sense) on some subspace of observables $(\Ecal, \norm{\cdot}_\Ecal)$, relative to a stationary measure $\eta$, then every observable in this subspace $\Ecal$ satisfies (exponential) large deviations. More precisely,  for all $n \ge 1$ consider the $n$-th stochastic sum
$$S_n \phi := \phi \circ Z_{n-1} + \cdots + \phi \circ Z_0 .$$    

Then for all $\ep > 0$, 
\begin{equation}\label{ldt abstract eq}
\Pp \left\{ \abs{ \frac{1}{n} S_n \phi - \int_M \phi \, d\eta } >\epsilon \right\} \le 8e^{-c(\epsilon) n}
\end{equation}
where $c (\ep)$ is an explicit rate function. This rate function is essentially of order $c (\ep) \sim \ep^2 \, \norm{\phi}_\Ecal^{-2}$. See~\cite{CDK-paper3} for more details, specifically Definition 1.1, Theorem 1.1, Remark 1.1. 

\smallskip

In this work, the  Markov operator $\bar \Qop$ is uniformly ergodic (hence strongly mixing)  on the entire space of observables $L^\infty (\Ascr \times \Pp^1)$ and relative to the $L^\infty$-norm.  
Moreover, for the remaining part of the argument,  it is crucially  important to note the precise dependence of the large deviations estimate on the norm (in our case the $L^\infty$-norm) of the observable. 

We are ready to state and prove a stronger, parametric version of the large deviations type (LDT) estimate for cocycles in $\cocycles$, which will also imply the result stated in the introduction.

\begin{theorem} 
\label{Parametric LDT} 
Let $A \colon \Lambda \to \cocycles$ be a smooth family of cocycles satisfying Assumptions~\ref{assumptions}.
Then for every $\ep>0$ and $\tilde{\eta}$-a.e $(t,\hat{v})$, there exist $c_0 (\ep) > 0$ and $n_0 (\ep, t,\hat{v}) \in \N$ such that for every $n \geq n_0 (\ep, t, \hat v)$, 
$$
\mu \left\{ \om \in X \colon \abs{  \frac{1}{n} \log \norm{A_t^n (\om) v}  - L_1 (A_t)} >\ep   \right\} < e^{-c_0 (\ep) n^{1/3}} \, .
$$
Moreover, for Lebesgue a.e. $t \in \Lambda$, given any $\ep > 0$ there are $c_0 (\ep) > 0$ and $n_0 (\ep, t) \in \N$ such that for all $n\ge n_0 (\ep, t)$.
 $$
\mu \left\{ \om \in X \colon \abs{  \frac{1}{n} \log \norm{A_t^n (\om) }  - L_1 (A_t)} >\ep   \right\} < e^{-c_0 (\ep) n^{1/3}} \, ,
$$
that is, a (sub-exponential) large deviations type estimate holds for Lebesgue almost every cocycle along the curve $t \mapsto A_t$.
\end{theorem}

\begin{proof}
Given a parameter $t\in\Lambda$, consider the Markov chain $Z_n \colon X \times \Pp^1 \to \Ascr \times \Pp^1$,
$$Z_n (\om, \hat v) := (\om_{n+1}, \hat A_t^n (\om) \hat v) \, .$$

It is easy to see that its transition kernel is given by
$$\Ascr \times \Pp^1 \ni (\om_1, \hat v) \mapsto p \times \delta_{\hat A_t (\om_1) \hat v} ,$$
so that its corresponding Markov operator is precisely the operator $\bar \Qop_t$ defined above, whose stationary measure is $p \times \eta_t$. Recall that by Corollary~\ref{cor strong mixing}, $\bar \Qop_t$ is strongly mixing on $(L^\infty (\Ascr \times \Pp^1), \norm{\cdot}_\infty)$. 

For the observable $\varphi_t(i,\hat{v}) := \log \norm{A_t(i) \frac{v}{\norm{v}}}$ we have that 
\begin{align*}
\varphi_t \circ Z_{n-1} (\om, \hat v) = \varphi_t (\om_n, \hat A_t^{n-1} (\om) \hat v) 
& = \log \norm{  A_t (\om_n) \, \frac{ \hat A_t^{n-1} (\om) \hat v }{\norm{  \hat A_t^{n-1} (\om) \hat v }  }   } \\
& = \log \norm{ \hat A_t^{n} (\om) \hat v }  -  \log \norm{ \hat A_t^{n-1} (\om) \hat v }   \, .
\end{align*}
Thus
$$\frac{1}{n} \, S_n \varphi_t (\om, \hat v) =  \frac{1}{n} \log \norm{A_t^n (\om) v} $$
where $v$ is a unit vector representative of the point $\hat v$.

Moreover, by Furstenberg's formula~\eqref{F formula 10},
$$L_1 (A_t) = \int \psi_t (\hat v) \, d \eta_t (\hat v) = \int \varphi_t (i, \hat v) \, d (p\times\eta_t) (i, \hat v) .$$
It is not difficult to verify  that Lemma~\ref{lem exp small measure} item (iii) implies that for Lebesgue a.e. $t\in\Lambda$, $L_1 (A_t) > - \infty$.

Our statement (in fact, a stronger version thereof) would then immediately follow  from~\eqref{ldt abstract eq} if the observable $\varphi_t$ were bounded. That is of course not the case, precisely because of the singular matrices. The idea is then to use the truncations $\varphi_{t, N}$ (for which estimate~\eqref{ldt abstract eq} is applicable) as a substitute for $\varphi_t$, where the order $N$ of the truncation is adapted to the scale $n$ at which we prove the LDT estimate.  

\smallskip

More precisely, given any large enough scale $n\in\N$, let $N_n \in \N$ be a truncation order to be chosen later (anticipating, the choice that optimizes the estimates will turn out to be $N_n \sim n^{1/3}$). We will transfer the LDT estimate at scale $n$  from $\varphi_{t, N_n}$  to $\varphi_t$ by eliminating an $\tilde \eta$-negligible set of parameters $(t, \hat v)$. 

For all $n\in \N$, define the following real-valued functions on $X \times \Lambda \times \Pp^1$:
\begin{align*}
\Delta_n (\om, t , \hat v) & := \frac{1}{n} S_n \varphi_t (\om, \hat v) - \int  \varphi_t \, d p \times \eta_t  \\
E_n (\om, t , \hat v) & := \frac{1}{n} S_n \varphi_{t, N_n} (\om, \hat v) - \int  \varphi_{t, N_n} \, d p \times \eta_t\\
g_n (\om, t, \hat v) & := \frac{1}{n} S_n \varphi_t (\om, \hat v) - \frac{1}{n} S_n \varphi_{t, N_n} (\om, \hat v) .
\end{align*}

Then
\begin{equation}\label{super duper eq}
\abs{ \Delta_n (\om, t , \hat v)  - E_n (\om, t , \hat v)  } \le \abs{ g_n (\om, t, \hat v) } + \int  \abs{ \varphi_t - \varphi_{t, N_n} } \, d p \times \eta_t . 
\end{equation}

Note that 
$$\Delta_n (\om, t , \hat v)  = \frac{1}{n} \log \norm{A_t^n (\om) v} - L_1 (A_t) \, .$$

Moreover, given $\ep > 0$, the general LDT estimate~\eqref{ldt abstract eq} and the comments following it show that  {\em for every} $(t, \hat v) \in \Lambda \times \Pp^1$, 
\begin{equation}\label{ldt truncation}
\mu \left\{ \om \in X \colon \abs{ E_n (\om, t , \hat v) } > \frac \ep 2 \right\} \le 8  \, e^{- c_0 (\ep) N_n^{-2} \, n}
\end{equation}
where $c_0 (\ep)$ is essentially of order $\ep^2$.
Using Lemma~\ref{lemma Estimates phi and psi} we have
\begin{align*}
&\int \left(  \int  \abs{ \varphi_t - \varphi_{t, N_n} } \, d p \times \eta_t \right) \, dt \\
&= \sum_{i\in\Ascr} p_i \int_{\Lambda} \int_{\Pp^1} \abs{ \varphi_t (i,  \hat v) -  \varphi_{t, N_n} (i,  \hat v)} \, d\eta_t (\hat{v}) \, dt \less e^{-N_n/3} \, .
\end{align*}

Moreover, we will also show that 
\begin{equation}\label{eq 101010}
\iint \abs{ g_n (\om, t, \hat v) } \, d \mu (\om) \, d \tilde \eta (t, \hat v) \less \sqrt{n} \, e^{-N_n/3} \, .
\end{equation}

Indeed, note that 
$$g_n (\om, t, \hat v) = \frac{1}{n} \sum_{j=1}^{n} \left( \varphi_t (\om_j, \hat A_t^{j-1} (\om) \hat v) -  \varphi_{t, N_n} (\om_j, \hat A_t^{j-1} (\om) \hat v) \right) \, .$$

Given $\om\in X$, if $g_n (\om, t, \hat v) \neq 0$ then there is $1\le j \le n$ such that $\varphi_t (\om_j, \hat A_t^{j-1} (\om) \hat v) \neq  \varphi_{t, N_n} (\om_j, \hat A_t^{j-1} (\om) \hat v)$, that is, $\varphi_t (\om_j, \hat A_t^{j-1} (\om) \hat v) < - N_n$, or $(t, \hat v) \in P_{\om_j, j-1, N_n} (\om)$. 

Thus for $(t, \hat v) \notin B_n (\om) := \bigcup_{j=1}^n P_{\om_j, j-1, N_n} (\om)$, where $\tilde \eta (B_n (\om) \less n e^{- N_n}$, we have that $g_n (\om, t, \hat v) = 0$. 

Using Fubini and Cauchy-Schwarz it follows that
\begin{align*}
& \iint \abs{ g_n (\om, t, \hat v) } \, d \mu (\om) \, d \tilde \eta (t, \hat v) 
 = \iint \abs{ g_n (\om, t, \hat v) } \, d \tilde \eta (t, \hat v) \, d \mu (\om) \\
& = \iint \abs{ g_n (\om, t, \hat v) } \, \ind_{B_n (\om)} (t, \hat v) \, d \tilde \eta (t, \hat v) \, d \mu (\om)  \\
& \le \int \norm{g_n (\om, \cdot)}_{L^2 (\tilde \eta)} \, \, \tilde \eta (B_n (\om))^{1/2} \, d \mu (\om) \\
& \less \sqrt{n} \, e^{-N_n/2} \, \int \norm{g_n (\om, \cdot)}_{L^2 (\tilde \eta)} \, d \mu (\om) \, .
\end{align*}

By Lemma~\ref{lem exp small measure} item (iii), for all $\om\in X$, 
$$\int \varphi_t^2 (\om_j, \hat A_t^{j-1} (\om) \hat v) \, d \eta_t (\hat v) \, d t \le C ,$$
so 
$$\norm{ \frac{1}{n} S_n \varphi_t (\om, \hat v) }_{L^2 (\tilde \eta)} \le \sqrt{C} , $$
while $\abs{ \varphi_{t, N_n} (i, \hat w)} \le N$ so
$$\norm{ \frac{1}{n} S_n \varphi_{t, N_n} (\om, \hat v) }_{L^2 (\tilde \eta)} \le N_n .$$

Thus for all $\om \in X$, $\norm{g_n (\om, \cdot)}_{L^2 (\tilde \eta)} \le  \sqrt{C} + N_n$, which then easily implies~\eqref{eq 101010}.

Integrating the inequality~\ref{super duper eq} with respect to $\mu$ and $\tilde \eta$ we obtain
\begin{equation*}\label{super duper eq2}
\iint \abs{ \Delta_n (\om, t , \hat v)  - E_n (\om, t , \hat v)  } \, d \mu (\om) \, d \tilde \eta (t, \hat v) \less \sqrt{n} \, e^{-N_n/3} \, .
\end{equation*}

Applying Chebyshev's inequality to the function 
$$(t, \hat v) \mapsto \int \abs{ \Delta_n (\om, t , \hat v)  - E_n (\om, t , \hat v)  } \, d \mu (\om) ,$$ 
there is a set $\B_n \subset \Lambda \times \Pp^1$ such that $\tilde \eta (\B_n) \less \sqrt{n} \, e^{-N_n/6}$ and if $(t, \hat v) \notin \B_n$ then
 $$\int \abs{ \Delta_n (\om, t , \hat v)  - E_n (\om, t , \hat v)  } \, d \mu (\om) < e^{-N_n/6} \, .$$ 

Since $\sum_{n\ge1} \tilde \eta (\B_n) < \infty$, by Borel-Cantelli, $\tilde \eta$-almost every $(t, \hat v)$ belongs to a finite number of sets $\B_n$. That is, for  $\tilde \eta$-a.e. $(t, \hat v)$ there is $n_0 (t, \hat v) \in \N$ such that for all $n\ge n_0 (t, \hat v)$ we have $(t, \hat v) \notin \B_n$, so 
$$\int \abs{ \Delta_n (\om, t , \hat v)  - E_n (\om, t , \hat v)  } \, d \mu (\om) < e^{-N_n/6} \, .$$ 

Applying again Chebyshev we have that
$$\mu \left\{ \om\in X \colon \abs{ \Delta_n (\om, t , \hat v)  - E_n (\om, t , \hat v)  } > e^{-N_n/12} \right\} 
\le e^{-N_n/12} .$$

We have shown that there is a set $\Omega_n \subset X$ with $\mu (\Omega_n) < e^{-N_n/12}$ such that if $\om \notin \Omega_n$ then
$$ \abs{ \Delta_n (\om, t , \hat v)  - E_n (\om, t , \hat v)  }  < e^{-N_n/12} \le \frac \ep 2 $$
provided $n$ is large enough depending on $(t, \hat v)$ as above and on $\ep$.

\smallskip

Thus if $ \abs{ \Delta_n (\om, t , \hat v)} > \ep$ then either $\om\in\Omega_n$ or  $\abs{ E_n (\om, t , \hat v)} > \frac \ep 2$, the later holding for $\om$ in a set of measure $\le 8  \, e^{- c_0 (\ep) N_n^{-2} \, n}$ by~\eqref{ldt truncation}. 

We conclude that for $\tilde \eta$-a.e.$(t, \hat v)$ and for all $\ep > 0$ there is $n_0 = n_0 (\ep, t, \hat v) \in \N$ such that for all $n\ge n_0$ we have 
$$\mu \left\{ \om \colon \abs{ \Delta_n (\om, t , \hat v)} > \ep \right\} < e^{-N_n/12} + 8  \, e^{- c_0 (\ep) N_n^{-2} n}
<   e^{- c_0 (\ep) \, n^{1/3} }
$$
provided we choose $N_n \sim n^{1/3}$. This establishes the first statement of the theorem.

In particular, by the definition of the measure $\tilde \eta$, for Lebesgue almost every parameter $t\in\Lambda$ and for $\eta_t$-a.e. point $\hat v$  this first statement of the theorem holds, namely
\begin{equation}\label{ldt vectors}
\mu \Big\{ \om \in X \colon \abs{  \frac{1}{n} \log \norm{A_t^n (\om) v}  - L_1 (A_t)} >\ep   \Big\} < e^{-c_0 (\ep) n^{1/3}} \, .
\end{equation}

The support of the measure $\eta_t$ is the set 
$$\left\{ \hat A_t^n (s\uom) \hat r_s \colon s \in \Asing, \, \uom \in \Ainv^n, \, n\in \N \right\} .$$

There is at least one singular symbol, $s \in \Asing$ and at least one invertible one $i \in \Ainv$. By the positive winding condition,
$$\frac{d}{dt} \hat A_t (i) \, \hat r_s \ge c_0 > 0 ,$$
so $t \mapsto \hat A_t (i) \, \hat r_s$ cannot be constant on a set of positive Lebesgue measure. Therefore 
$\hat A_t (i) \, \hat r_s \neq \hat r_s$ for Lebesgue almost every $t$.

It follows that for Lebesgue almost every $t\in\Lambda$, the support of the measure $\eta_t$ has at least two points, $\hat e_t^1$ and $\hat e_t^2$. Thus~\eqref{ldt vectors} holds for two linearly independent vectors $e_t^1$ and $e_t^2$ so it must hold  with the matrix norm $ \norm{A_t^n (\om)}$, which establishes the second statement of the theorem.
\end{proof}

Let us now derive the version of the LDT estimate stated in the introduction.

\begin{proof}[Proof of Theorem~\ref{thm: ldt intro}] Recall that given the partition of the alphabet  $\Ascr = \Asing \sqcup \Ainv$ into two nonempty sets, we denote by
$\Mscr$ the set of all cocycles $\uA = (A_1, \ldots, A_k) \in \Mat_2(\R)^\Ascr$ with  $\rank A_i = 1$  if  $i\in\Asing$ and $\det(A_j) >  0$ if $j\in\Ainv$. Moreover, $\Mscr^*$ denotes the subset of cocycles in $\Mscr$ for which $\hat r_i \neq \hat k_j$ for all $i, j \in \Asing$.
 
Firstly note that the set $\Mscr \setminus \cocycles$, consisting of cocycles for which a kernel $\hat r_i$ coincides with a range $\hat k_j$ for some symbols $i, j \in \Asing$, has zero Lebesgue measure in the analytic manifold $\Mscr$.
Secondly, by~\cite[Corollary 4.5]{DDGK-paper1}, $L_1 (\uA) > - \infty$ for Lebesgue almost every $\uA\in\Mscr$.\footnote{Note that the partition $\Ascr=\Asing\sqcup\Ainv$ in this paper is fixed, and we use the notation $\Mscr$ to refer to the space that  in~\cite{DDGK-paper1}  was denoted by  $\Mscr_I$ with $I=\Asing$.}
We may then neglect cocycles in these negligible sets. 

Given any cocycle $\uA\in\Mscr$,  for all parameters $t\in [-\pi, \pi]$
define
$\uA_t:=(A_t (1), \ldots, A_t (k))$, where
$$ A_t (i):= \begin{cases}
	A_i & \text{ if }  \, i \in \Asing \\
A_i \, R_t &\text{ if } \, i \in \Ainv
\end{cases} \quad \text{ and } \quad R_t:= \begin{bmatrix}
\cos t & -\sin t \\ \sin t & \cos t 
\end{bmatrix} . $$

This is an analytic family of cocycles. If $\uA\in\cocycles$ then $\uA_t$ takes values in $\cocycles$ and it satisfies the Assumptions~\ref{assumptions}; then by Theorem~\ref{Parametric LDT}, $\uA(t)$ satisfies large deviations of sub-exponential type for Lebesgue almost every $t\in \R$.
Therefore, since the map $\Mscr \times [-\pi,  \pi] \to \Mscr$,
$(\uA, t)\mapsto \uA_t$ is a submersion, Lebesgue almost every
$\uA\in \Mscr$ satisfies similar large deviations estimates.
\end{proof}

\begin{subsection}{Central limit theorem}

We now establish a central limit theorem for cocycles with singular components.  The proof uses the following abstract central limit theorem due to Gordin-Liv\v sic~\cite{Go78}.

Let $M$ be a compact metric space, let $M\ni x \mapsto K_x \in \Prob(M)$ be an ergodic\footnote{This automatically hods when there is a unique stationary measure.} transition kernel, let $Q$ be the corresponding Markov operator and let $\eta$ be a $Q$-stationary measure.   
Given an observable $\varphi \in L^2 (M, \eta)$ with $\int \varphi \, d \eta =0$, if  $ \sum_{n=0}^\infty \norm{Q^n \varphi}_2 < \infty$, we can define $g \in L^2(M, \eta)$ by $g := \sum_{n=0}^\infty Q^n \varphi$. 
Then $\varphi = g - Q g$. Let $\sigma^2(\varphi) := \norm{g}_2^2 - \norm{Q g}^2_2 \in [0, \infty)$.

\begin{theorem} \label{thm Abstract CLT} (Gordin-Liv\v sic)
Let $M, Q, \eta, \varphi$ be as above and  assume that 
$$
\sum_{n=0}^\infty \norm{Q^n\varphi}_2 < \infty \quad \text{and} \quad \sigma^2(\varphi) > 0.
$$
Then we have the following convergence to the standard normal distribution:
$$
\frac{S_n\varphi}{\sigma \sqrt{n}} \to \mathcal{N}(0,1).
$$
Moreover, this holds not only relative to the Markov probability with initial distribution $\eta$, but also with initial distribution $\delta_x$ for $\eta$-a.e. $x$.
\end{theorem}

We consider as in the previous subsection a smooth family of cocycles $\Lambda \ni t \mapsto \uA_t \in \cocycles$ satisfying Assumptions~\ref{assumptions}. We also consider the special observables $\psi_t$ and $\varphi_t$ from before. 

We introduced two related (families of) Markov operators acting on measurable bounded observables. They can actually be defined for arbitrary observables, namely if $\psi \colon \Pp^1  \to [-\infty, \infty)$ we put
$$\Qop_t \psi (\hat{v}) = \sum_{i \in \Ascr} p_i \psi(\hat{A}_t(i)\hat{v}) $$
and if $\varphi \colon \Ascr \times \Pp^1  \to [-\infty, \infty)$ we put
$$\bar{\Qop}_t \varphi(j,\hat{v}) = \sum_{i \in \Ascr}  p_i \varphi(i, \hat{A}_t(j) \hat{v}) .$$

Consider the projection:
\begin{align*}
&\pi \colon L^2(\Ascr \times \Pp^1) \to L^2(\Pp^1) \\ 
& \pi \varphi (\hat{v}) := \sum_{i \in \Ascr} p_i \varphi (i,\hat{v}).
\end{align*} 
Using this notation, $\pi \varphi_t = \psi_t$. Moreover, $\bar{\Qop}^{n+1}_t \varphi(i,\hat{v}) = \Qop^n_t (\pi \varphi)(\hat{A}_t(i)\hat{v})$.


\begin{lemma} \label{last lemma} For Lebesgue almost every parameter $t\in\Lambda$ we have that
$$
\sum_{n=0}^\infty \norm{ \bar{\Qop}^{n+1}_t\varphi_t - \int \varphi_t \; dp d\eta_t }_{L^2(p \times \eta_t)} < \infty.
$$
\end{lemma}
\begin{proof}
By the previous observations,
\begin{align*}
\sum_{n=0}^\infty \norm{ \bar{\Qop}^{n+1}_t\varphi_t - \int \varphi_t \; dp d\eta_t }_{L^2} &= \sum_{n=0}^\infty \norm{ \Qop^n_t \psi_t(\hat{A}_t(i)\hat{v}) - \int \psi_t \; d\eta_t }_{L^2} \\
&\kern-4em = \sum_{n=0}^\infty \left( \sum_{i \in \Ascr} p_i  \int \left| \Qop^n_t \psi_t(\hat{A}_t(i)\hat{v}) - \int \psi_t \; d\eta_t \right|^2 \; d\eta_t(\hat{v}) \right)^{\frac{1}{2}}.
\end{align*}
We start by showing that for every $i \in \Ascr$ there exists $\sigma \in (0,1)$ such that for every $n \in \N$
$$
\int_{\Lambda} \int_{\Pp^1}  \left| \Qop^n_t \psi_t(\hat{A}_t(i)\hat{v}) - \int \psi_t \; d\eta_t \right|^2 \; d\eta_t(\hat{v}) \;dt \lesssim \sigma^n.
$$ 
Fix $i \in \Ascr$ and note that by lemma~\ref{lem exp small measure} item (iii) and Jensen's inequality,
\begin{align*}
&\int_{\Lambda} \int_{\Pp^1}  \left| \Qop^n_t \psi_t(\hat{A}_t(i)\hat{v})  - \int \psi_t \; d\eta_t \right|^2 \; d\eta_t(\hat{v}) \;dt \\
&= \int_{\Lambda} \int_{\Pp^1} \Big| \sum_{\uom \in \Ainv^n} p(\uom) \psi_t(\hat{A}^n_t(\uom)\hat{A}_t(i)\hat{v}) \\
& - \sum_{s \in \Asing}p_s \sum_{j=n}^\infty \sum_{\uom \in \Ainv^j} p(\uom) \psi_t(\hat{A}_t^j(\uom)r_s)  \Big|^2d\eta_t(\hat{v})dt \\
&\leq 2 \int_{\Lambda} \int_{\Pp^1} \sum_{\uom \in \Ainv^n} p(\uom) \left| \psi_t(\hat{A}^n_t(\uom)\hat{A}_t(i)\hat{v}) \right| ^2 \; d\tilde{\eta}(t,\hat{v})  \\
&+ 2 \int_{\Lambda} \int_{\Pp^1} \sum_{s \in \Asing}p_s \sum_{j=n}^\infty \sum_{\uom \in \Ainv^j} p(\uom) \left| \psi_t(\hat{A}_t^j(\uom)r_s) \right| ^2  \; d\tilde{\eta}(t,\hat{v}) \\
&\less \sum_{\uom \in \Ainv^n} p(\uom)  +  \sum_{s \in \Asing}p_s \sum_{j=n}^\infty \sum_{\uom \in \Ainv^j} p(\uom)  \less (1-q)^n 
\end{align*}
so the claim holds with $\sigma=1-q$.

By Chebyshev's inequality, 
$$
\Leb \left\{ t \in \Lambda \colon \int_{\Pp^1} \left|  \Qop^n_t \psi_t(\hat{A}_t(i)\hat{v}) - \int \psi_t \; d\eta_t  \right|^2 \;d\eta_t (\hat{v}) > \sigma^\frac{n}{2}   \right\} \less \sigma^\frac{n}{2}.
$$

Therefore, for every $i \in \Ascr$ and for each $n \in \N$, there exists a set $B_n(i) \subset \Lambda$, with $\Leb(B_n(i)) \less\sigma^{\frac{n}{2}}$ such that for every $t \notin B_n(i)$,
\begin{equation} \label{Eq Gordin Livsic}
\norm{\bar{\Qop}_t^{n+1} \varphi_t - \int_{\Ascr \times \Pp^1} \varphi_t \; d p d\eta_t}_{L^2(p \times \eta_t)} \lesssim \sigma^{\frac{n}{2}}.
\end{equation}

Moreover, since there is a finite number of symbols, there exists a set $B_n$ that satisfies the same properties for all symbols $i\in\Ascr$ simultaneously. Furthermore, since $\sum_{n=0}^\infty \Leb(B_n) < \infty$, by the Borel-Cantelli lemma, for almost every $t \in \Lambda$, there exists $n_0(t) \in \N$ such that for every $n \geq n_0(t)$, $t \notin B_n$. Thus for almost every $t$, the inequality \eqref{Eq Gordin Livsic} holds, hence
$$
\sum_{n=0}^\infty \norm{ \bar{\Qop}^{n+1}_t\varphi_t - \int \varphi_t \; d p d\eta_t }_{L^2(p \times \eta_t)} < \infty ,
$$
which proves the lemma.
\end{proof}

\medskip

We are ready to formulate and to prove a parametric version of the CLT for cocycles with singular components.

\begin{theorem}\label{Parametric CLT}
Let $A \colon \Lambda \to \cocycles$ be a smooth family of cocycles satisfying Assumptions~\ref{assumptions}. Then for almost every $t \in \Lambda$ there exists $\sigma=\sigma(t)>0$ such that 
$$
\frac{\log \norm{A_t^n v}-n\, L_1(A_t)}{\sigma\, \sqrt{n}} \stackrel{d}{\to} \mathcal{N}(0,1)
$$
holds for $\eta_t$-a.e. $\hat v \in \Pp^1$ and, moreover, 
$$
\frac{\log \norm{A_t^n\, }-n\, L_1(A_t)}{\sigma\, \sqrt{n}} \stackrel{d}{\to} \mathcal{N}(0,1) .
$$
\end{theorem}

\begin{proof}
If $A_t$ is projectively uniformly hyperbolic for every $t\in \Lambda$ then  the cocycle reduces to a $1$-dimensional cocycle to which the classical CLT applies. Hence we assume from now on that the cocycle family $A_t$ is not projectively uniformly hyperbolic for some $t_0\in \Lambda$. By  Theorem~\ref{ManeBochiMat2} there exist parameters $\Lambda\ni t_n\to t_0$ for which $A_{t_n}$ admits a null word, which implies a singularity $L_1(A_{t_n})=-\infty$.
 
 Consider the special observable $\varphi_t$ defined before and recall from Lemma~\ref{last lemma} that for Lebesgue almost every parameter $t\in\Lambda$ we have  
 $$\sum_{n=0}^\infty \norm{ \bar{\Qop}^{n}_t\varphi_t - \int \varphi_t \; dp d\eta_t }_{L^2(p \times \eta_t)} < \infty .$$  
 
This in particular allows us to define
$$
g_t := \sum_{n=0}^\infty \bar{\Qop}^n_t \left(\varphi_t - \int \varphi_t \; dp d\eta_t\right) \in L^2 (p\times\eta_t) .
$$

Then  for almost every parameter $t$,
$$\varphi_t - \int \varphi_t \; dp d\eta_t = g_t - \bar{\Qop}_t g_t .$$
 
 Let $\sigma^2 (t) = \sigma^2_t(\varphi_t) := \norm{g_t}_2^2 - \norm{\bar{\Qop}_t g_t}^2_2 \in [0, \infty)$.
 
 In order to apply Theorem~\ref{thm Abstract CLT}  it remains to prove that $\sigma^2 (t)>0$ for almost every $t \in \Lambda$. We accomplish this in several steps.

\smallskip

We recall the following definitions and results.  
A cocycle is said to be \emph{projectively uniformly hyperbolic} if it admits a dominated splitting~\cite[Definition~3.1]{DDGK-paper1}.  

In the Bernoulli case, an \emph{invariant multi-cone} is a non-empty, non-dense open subset 
$M \subset \Pp^1$ such that $\overline{A_i M} \subset M$ for every $i \in \Ascr$.  
In~\cite[Theorem~3.1]{DDGK-paper1} we proved that a cocycle $\uA$ is projectively uniformly hyperbolic if and only if it admits an invariant multi-cone.

Assume now that the invertible cocycle $\uA_{\mathrm{inv}} = (A_i)_{i \in \Ainv}$ is projectively uniformly hyperbolic, and let $M$ be an invariant multi-cone.  
Define
\begin{align*}
\Kinv^u &= \bigcap_{n=0}^{\infty} \bigcup_{i_1,\dots,i_n \in \Ainv} A_{i_n}\cdots A_{i_1}(M), \\
\Kinv^s &= \bigcap_{n=0}^{\infty} \bigcup_{i_1,\dots,i_n \in \Ainv} (A_{i_n}\cdots A_{i_1})^{-1}(\Pp^1 \setminus \overline{M}).
\end{align*}

These sets are, respectively, the closures of the unstable and stable Oseledets directions of $\uA_{\mathrm{inv}}$ (see~\cite[Proposition~3.1]{DDGK-paper1}).

\begin{definition}
	A smooth family of cocycles $\uA \colon \Lambda \to \Matdm^k$ is called \emph{rich} if for every $p_0 \in \Lambda$ there exists a smooth curve $t \mapsto \alpha(t) \in \Lambda$ through $p_0$ such that the one-parameter family $\uA(\alpha(t))$ is positively winding.
\end{definition}

The following theorem was proved in~\cite[Theorem~4.1]{DDGK-paper1}.

\begin{theorem} \label{ManeBochiMat2}
	Let $\uA \colon \Lambda \to \Matdm^k$ be a rich smooth family of cocycles on an open set $\Lambda \subseteq \R^m$, and suppose there exists $i \in \{1, \dots, k\}$ such that $\rank(A_i(t)) = 1$ for every $t \in \Lambda$.  
	Then, for each $t \in \Lambda$, either $\uA(t)$ is projectively uniformly hyperbolic, or there exists a sequence $t_n \to t$ in $\Lambda$ such that $\uA(t_n)$ has a periodic null word, hence $L_1(\uA(t_n)) = -\infty$ for all $n \in \N$.
\end{theorem}

For the Markov case, we refer to the companion paper~\cite{DDGK-paper1}, in particular Definitions~3.1 and~3.5, Proposition~3.1, and Theorem~1.1.

\begin{lemma}\label{prev lemma}
Let $\mathcal{H}$ be the set of parameters $t\in\Lambda$ such that the corresponding cocycle $\uA_t$ is projectively uniformly hyperbolic.  For Lebesgue almost every $t \in \mathcal{H}$, the corresponding stationary measure $p \times \eta_t$ has infinite support.
\end{lemma}

\begin{proof}
Firstly note that if the parameter $t \in \mathcal{H}$ is such that $\uA_t$ is 
diagona\-lizable, then the ranges of the singular matrices and the unstable directions of the invertible matrices are aligned. Therefore, the support of the corresponding stationary measure has only one point. However, this case only happens with zero measure, since this alignment will be undone by the winding property as the parameter $t$ varies. Note that singular matrices remain constant in the process.

Hence we only need to consider the case in which the cocycle $\uA_t$ is not diagonalizable. Note that if $\uA_t$ is projectively uniformly hyperbolic, then so is its invertible part $\uA_{{\rm inv}} (t) := (A_t (i))_{i\in \Ainv}$. Hence we can define $\Kinv^u$, the  set of  unstable Oseledets directions $E^u(\omega)$ of the cocycle $(A_i)_{i\in \Ainv}$ over the set of points $\omega\in \Ainv^\Z$. We divide the proof in two cases: $\# \Kinv^u = 1$ and $\# \Kinv^u >1$.

First, we suppose that $\# \Kinv^u = 1$. Since we assume that $\uA_t$ is not diagonalizable, there is no range in $\Kinv^u$. Therefore, the iterates of the ranges are going to spread into infinitely many different points in the projective space by the projective uniform hyperbolic dynamics (in fact, the iterates will converge to $K^u_{\rm{inv}}$). Since the corresponding stationary measure $p \times \eta_t$ is discrete and gives positive weight to every pair of the form $(j, \hat A^n_t(\uom) \hat r_s)$, we conclude that the support of $\eta_t$ is infinite. 

Now suppose that $\#\Kinv^u > 1$.  
In this case, $\Kinv^u$ is in fact infinite, as it is the limit set of a contracting iterated function system; it may be a Cantor set, an interval, or a combination of both.  

If any range lies outside $\Kinv^u$, the proof proceeds as in the previous case, so we may assume that every range $r_i$ is contained in $\Kinv^u$.  

Moreover, since $\uA_{\mathrm{inv}}$ is projectively uniformly hyperbolic (and topologically mixing in the Markov case), for any $\hat v \in \Kinv^u$ and any point in an invariant multi-cone of $\uA_{\mathrm{inv}}$,   in particular for any range $\hat r_i$,  there exists an admissible word $\omega \in \Ainv^\N$ such that
\[
\hat v = \lim_{n \to \infty} \hat A_t^n(\omega)\, \hat r_i.
\]
Hence, by a suitable choice of $\omega$, one can iterate any range $r_i$ by $A_t^n(\omega)$ so as to converge to any prescribed element of $\Kinv^u$. 
Since the corresponding stationary measure $p \times \eta_t$ is discrete and gives positive weight to every pair of the form $(j, \hat A^n_t(\uom) \hat r_s)$, we conclude that the support of $\eta_t$ is infinite. 
\end{proof} 

\begin{lemma}\label{lemma observable not constant}
For Lebesgue almost every parameter $t \in \Lambda$, the obser\-vable $\varphi_t$  cannot be $p \times \eta_t$ constant. 
\end{lemma}

\begin{proof}
By Theorem~\ref{ManeBochiMat2},  for every $t \in \Lambda$,  either $\uA_t$ is projectively uniformly hyperbolic or it can be approximated by cocycles that admit null words. The proof is then divided into two cases. 

First let us consider the set $\mathcal{B}$ of parameters $t \in \Lambda$ such that $\uA_t$ is not projectively uniformly hyperbolic. Note that there exists $c>0$ such that for every $t \in \mathcal{B}$ and all $j \in \Ascr$ and $i \in \Asing$,  it holds that $|\varphi_t(j, \hat{r}_i)| \leq c$. Moreover, by Theorem~\ref{ManeBochiMat2} there exists $t'$ arbitrarily close to $t$ such that $\uA_{t'}$ has a null word. Hence, by continuity, there exists some range $r_s$ and a finite word $\uom$ such that $\norm{A_t^n(\uom)r_s}$ is arbitrarily small, thus $|\varphi_t| \gg c$. Therefore, we conclude that for every $t \in \mathcal{B}$, $\varphi_t$ is not constant.

By Lemma~\ref{prev lemma} we have that for Lebesgue almost every $t \in \mathcal{H}$, there are infinitely many points in the support of the corresponding stationary measure $p \times \eta_t$. Moreover, $\varphi_t$ is $p \times \eta_t$-constant if and only if it takes the same value at  every point in the support of $p \times \eta_t$. The only way for this to happen is if $\uA_t$ had a conformal word (with a pair of non real eigenvalues). In fact, the presence of a conformal word implies that this word is projectively elliptic, hence not projectively uniformly hyperbolic. Therefore, we conclude that for almost every $t \in \mathcal{H}$, the observable $\varphi_t$ is not $p \times \eta_t$ constant, which finishes the proof.
\end{proof}

The next lemma establishes the positivity of the variance type quantity $\sigma^2 (t)$ for almost every $t \in \Lambda$ by an adaptation of the proof of~\cite[Proposition 2.2]{CDK-paper3}. We note that this proposition is a version of the abstract CLT of Gordin and Liv\v sic, which is more applicable to dynamical systems. 

\begin{lemma}\label{lemma positive variance}
For almost every $t \in \Lambda$ we have $\sigma^2(t)>0$. 
\end{lemma}
\begin{proof}
Consider the family of Markov operators $\bar{\Qop}_t \colon L^\infty(\Ascr \times \Pp^1) \to L^\infty(\Ascr \times \Pp^1) $, given by $(\bar{\Qop}_t \varphi)(j,\hat v) = \sum_{ i\in\Ascr}  \varphi(i,\hat A_t(j)\, \hat v)\, p_{i}$ and the corres\-ponding family of Markov kernels $\bar{K}_t \colon \Ascr \times \Pp^1 \to \Prob(\Ascr \times \Pp^1)$. Recall that the measure $p \times \eta_t$ is $\bar{K}_t$-stationary, in the sense that $(\bar{K}_t)_* p \times \eta_t =p \times  \eta_t$.

Fix any $t \in \Lambda$. Assume by contradiction that $\sigma_t^2(\varphi)=0$. Then
	\begin{align*}
		0 &\leq \int (\bar{\Qop}_t \, g_t (x) - g_t(y))^2 \; d\bar{K}_{t,x}(y) dp \times \eta_t(x) \\
		&= \int \left\{ (\bar{\Qop}_t \, g_t (x))^2 + (g_t(y))^2 -2 g_t(y)\bar{\Qop}_t g_t (x)\right\} \; d\bar{K}_{t,x}(y) dp \times \eta_t(x) \\
		& = \int \left\{   (g_t(y))^2 - (\bar{\Qop}_t \, g_t (x))^2\right\} \; d\bar{K}_{t,x}(y) dp \times \eta_t(x) \\
		&=\norm{g_t}_2^2 - \norm{\bar{\Qop}_t \, g _t}_2^2 = \sigma^2 (t) = 0.
	\end{align*}
	
We then conclude that $g_t(y) = \bar{\Qop}_t g_t(x)$ for $p \times \eta_t$-a.e $x$ and $\bar{K}_{t,x}$-a.e $y$. 
In particular, define
\[
X_j(\omega) := \hat A_t^j(\omega)\, \hat r_i,
\]
where $\omega \in \Ascr^\N$ is a periodic admissible word with 
$\omega_0 = \omega_m = i \in \Asing$. Then
\begin{align*}
	\log \norm{A_t^n(\omega)\, r_i} - n\, L_1(\uA_t) 
	&= \sum_{j=0}^{n-1} 
	\left[ \varphi_t\bigl(X_j(\omega)\bigr) 
	- \int \varphi_t \, d p \, d \eta_t \right] \\
	&= \sum_{j=0}^{n-1} 
	\Bigl[g_t\bigl(X_j(\omega)\bigr) - (\bar \Qop_t g_t)\bigl(X_j(\omega)\bigr)\Bigr] \\
	&= \sum_{j=0}^{n-1} 
	\Bigl[g_t\bigl(X_j(\omega)\bigr) - g_t\bigl(X_{j+1}(\omega)\bigr)\Bigr] \\
	&= g_t\bigl(X_0(\omega)\bigr) - g_t\bigl(X_n(\omega)\bigr) \\
	&= 0 .
\end{align*}
This telescoping sum shows that for every $n$-periodic orbit $\omega$ containing the singular symbol $i \in \Asing$ its dominant eigenvalue
satisfies
\[
|\lambda_{\text{max}}\bigl(A_t^n(\omega)\bigr)| = e^{n\, L_1(\uA_t)} .
\]
In particular, for the fixed point $\omega \equiv i$, 
\[
 \beta:= \bigl|\langle r_i, A_i\, r_i \rangle\bigr|= e^{ L_1(\uA_t) }
\]
 is independent of the parameter $t$.
Consequently, for any admissible $n$-periodic orbit $\omega$ with $\omega_0 = \omega_n = i$,
\[
\log \bigl|\langle r_i, A_t^n(\omega)\, r_i \rangle\bigr| = n\,\log \beta .
\]
Observe that the left-hand side defines a non-constant harmonic function of $t$, except possibly at its singularities where vanishing words occur 
(see the beginning of this proof).  Therefore, the above equality can only be satisfied for a discrete set of parameter values $t \in \Lambda$. For all other values of $t$, the equality fails, yielding a contradiction. It follows that for a.e. $t\in\Lambda$,
\[
\sigma^2(t) > 0,
\]
as claimed.
\end{proof}

Theorem~\ref{thm Abstract CLT} is then applicable to the Markov kernel $\bar K_t$ on  $\Ascr\times\Pp^1$ and the observable $\varphi_t$, for Lebesgue almost every $t\in\Lambda$. We then conclude that for $\eta_t$-a.e. point $\hat v \in \Pp^1$, 
$$
\frac{\log \norm{A_t^n\, v}-n\, L_1(A_t)}{\sigma\, \sqrt{n}} \stackrel{d}{\to} \mathcal{N}(0,1) ,
$$
thus establishing the first statement of the theorem.

To prove the second statement, with $\norm{A_t^n\, v}$ replaced by $\norm{A_t^n}$, we choose a unit vector $v$ for which the CLT holds and note that

$$  \frac{\log \norm{A_t^n}-n\, L_1(A_t)}{\sigma\, \sqrt{n}}  = \frac{\log \norm{A_t^n\, v}-n\, L_1(A_t)}{\sigma\, \sqrt{n}} 
+ \frac{\log \norm{A_t^n} - \log \norm{A_t^n\, v}}{\sigma\, \sqrt{n}} \, .$$

We claim that the sequence $\log \norm{A_t^n} - \log \norm{A_t^n\, v}$ is almost surely bounded, hence the last term above converges to $0$ almost surely, which would then conclude the argument via Slutsky's theorem. 

The claim clearly follows from the a.s. boundedness from below by a positive constant of the sequence $\frac{\norm{A_t^n\, v}}{\norm{A_t^n}} $. 
Let $\bar \mostv (A)$ denote the most expanding singular direction of the matrix $A$. By the Pythagorean's theorem,
$$\frac{\norm{A_t^n (\om) \, v}}{\norm{A_t^n (\om)}}  \ge \abs{ v \cdot \bar \mostv (A_t^n (\om))} \to \abs{ v \cdot \bar \mostv^\infty (A_t) (\om) } .$$

The limit direction $\bar \mostv^\infty (A_t) (\om)$ exists almost surely by arguments in the proof of Oseledets' theorem (see for instance~\cite[Proposition 4.4]{DK-book}) and $\bar \mostv^\infty (A_t) (\om)^\perp = E^s (\om)$, the stable subspace in the Oseledets theorem (see for instance the beginning of the proof of~\cite[Theorem 4.4]{DK-book}).  

We must have that $\abs{ v \cdot \bar \mostv^\infty (A_t) (\om) } > 0$ almost surely (which then establishes the claim). Otherwise, on a set of positive measure we would have that  $v \in \bar \mostv^\infty (A_t) (\om)^\perp = E^s (\om)$, which is not possible. Indeed, $E^s (\om)$ consists of the pre-images of the {\em kernel} $k_{\om_n}$ via matrix products $A_t^n (\om)$, for some $n \in \N$ such that $\om_n \in \Asing$ and $\om_0, \dots, \om_{n-1}\in \Ainv$. On the other hand, $\hat v$ is in the support of $\eta_t$, which consists of images of {\em ranges} $\hat r_s$ via matrix products $A_t^n (y)$. Thus $v \in E^s (\om)$ would imply the existence of null words, and in particular the fact that $L_1 (\uA_t) = - \infty$. But those parameters $t$ form a zero measure set.  
\end{proof}

Theorem~\ref{thm: clt intro}, the non-parametric version of the CLT stated in the introduction, can be derived form the parametric version exactly the same way the non-parametric LDT theorem~\ref{thm: ldt intro} was derived from the parametric LDT theorem~\ref{Parametric LDT}. Thus Theorem~\ref{thm: clt intro} holds as well.

\begin{remark}
We note that the statistical properties derived above  are sensitive to perturbations of the cocycle, that is, the parameters appearing in these estimates are not locally uniform.
\end{remark}

\begin{remark}\label{ldt-clt rank=1}
Projectively uniformly hyperbolic cocycles automatically satisfy uniform LDT estimates and a CLT, since these properties  can be immediately reduced to the corresponding limit laws for i.i.d. additive processes in classical probabilities. In particular, by~\cite[Theorem 4.2]{DDGK-paper1}, if all components of $\uA \in \Mat_2 (\R)^k$ are singular and $L_1 (\uA) > - \infty$ then $\uA$ satisfies uniform LDT estimates and a CLT.
\end{remark}
\end{subsection}

\subsection*{Acknowledgments}

P.D. was partially supported by FCT - Funda\c{c}\~{a}o para a Ci\^{e}ncia e a Tecnologia, through
the projects   UIDB/04561/2020 and  PTDC/MAT-PUR/29126/2017.
M.D. was supported by a CNPq doctoral fellowship and  also by the Coordena\c{c}\~{a}o de Aperfei\c{c}oamento de Pessoal de N\'{i}vel Superior - Brasil (CAPES) - Finance Code 001.
T.G. was supported by FCT - Funda\c{c}\~{a}o para a Ci\^{e}ncia e a Tecnologia, through the projects UI/BD/152275/2021 and by CEMS.UL Ref. UIDP/04561/2020, DOI 10.54499/UIDP/04561/ 2020.
S.K. was suppor\-ted by the CNPq research grant 313777/2020-9 and  by the Coordena\c{c}\~ao de Aperfei\c{c}oamento de Pessoal de N\'ivel Superior - Brasil (CAPES) - Finance Code 001.

\bibliographystyle{amsplain}
\bibliography{bib}

\end{document}